\documentclass[11pt,reqno]{amsart}
\usepackage{amsmath}
\usepackage{amssymb}
\usepackage{amsthm}
\usepackage{color}
\usepackage{eucal}
\usepackage{tikz}
\usepackage{gastex}
\usepackage{stmaryrd}
\usepackage{caption,subcaption}
\usepackage{latexsym}
\usepackage{indentfirst}
\usepackage{graphicx,accents}

\usetikzlibrary{decorations.pathmorphing}
\usetikzlibrary{shapes.geometric}

\DeclareSymbolFont{rsfscript}{OMS}{rsfs}{m}{n}
\DeclareSymbolFontAlphabet{\mathrsfs}{rsfscript}

\numberwithin{equation}{section}
\newtheorem{prop}{Proposition}[section]

\newtheorem{lem}[prop]{Lemma}
\newtheorem{cor}[prop]{Corollary}

\def\Jc{\mathrel{\mathrsfs{J}}}
\def\Dc{\mathrel{\mathrsfs{D}}}
\def\Hc{\mathrel{\mathrsfs{H}}}
\def\Lc{\mathrel{\mathrsfs{L}}}
\def\Rc{\mathrel{\mathrsfs{R}}}
\def\Kc{\mathrel{\mathrsfs{K}}}
\def\Rcc{\mathrm{R}}
\def\Hcc{\mathrm{H}}
\def\Lcc{\mathrm{L}}
\def\Dcc{\mathrm{D}}

\def\Kcc{\mathrm{K}}

\def\ep{\epsilon} 
\def\up{\upsilon} 
\def\la{\lambda} 
\def\ka{\kappa} 
\def\be{\beta} 
\def\si{\sigma} 
\def\al{\alpha}

\def\ap{\approx}
\def\xr{\xrightarrow}

\def\cev{\overset{{}_{\shortleftarrow}}}
\def\cevl{\overset{{}_{\longleftarrow}}}
\def\cevm{\overset{{}_{\leftarrow}}}
\def\ol{\overline}
\def\c{\mathrm{c}}
\def\l{\mathrm{l}}
\newcommand{\oX}{\ol{X}}
\renewcommand{\iff}{if and only if}

\DeclareMathOperator{\ft}{FI}
\DeclareMathOperator{\wig}{WIG}
\DeclareMathOperator{\wgen}{WG}
\DeclareMathOperator{\La}{L}
\DeclareMathOperator{\Mr}{MR}
\DeclareMathOperator{\M}{M}
\DeclareMathOperator{\G}{G}
\DeclareMathOperator{\Ca}{C}

\def\ftx{\ft(X)}
\def\ftn{\ft_n}
\def\mft{\ft_m}
\def\ftt{\ft_2}
\def\ftxm{\ft^1(X)}
\def\ftnm{\ft_n^1}
\def\wg{\wig(X)}
\def\Gx{\G(X)}
\def\Gxp{\G(X)^+}
\def\Gix{\G_i(X)}
\def\mrx{\Mr(X)}
\def\mx{\M(X)}
\def\mxm{\M^1(X)}
\def\lx{\La(X)}
\def\lo{\La_1(X)}
\def\lop{\La_1^+(X)}
\def\lom{\La_1^-(X)}
\def\ltp{\La_2^+(X)}
\def\ltm{\La_2^-(X)}
\def\lt{\La_2(X)}
\def\cx{\Ca(X)}

\title[]{Regular semigroups weakly generated by idempotents}
\author{Lu\'\i s Oliveira}
\address{Departamento de Matem\'atica,
Faculdade de Ci\^encias da Universidade do Porto,
R. Campo Alegre, 687, 4169-007 Porto, Portugal}
\email{loliveir@fc.up.pt}

\begin{document}

\begin{abstract}
A regular semigroup is weakly generated by a set $X$ if it has no proper regular subsemigroups containing $X$. In this paper, we study the regular semigroups weakly generated by idempotents.  We show there exists a regular semigroup $\ftx$ weakly generated by $|X|$ idempotents such that all other regular semigroups weakly generated by $|X|$ idempotents are homomorphic images of $\ftx$. The semigroup $\ftx$ is defined by a presentation $\langle \Gx,\rho_e\cup\rho_s\rangle$ and its structure is studied. Although each of the sets $\Gx$, $\rho_e$, and $\rho_s$ is infinite for $|X|\geq 2$, we show that the word problem is decidable as each congruence class has a ``canonical form''. If $\ftn$ denotes $\ftx$ for $|X|=n$, we prove also that $\ftt$ contains copies of all $\ftn$ as subsemigroups. As a consequence, we conclude that $(i)$ all regular semigroups weakly generated by a finite set of idempotents, which include all finitely idempotent generated regular semigroups, strongly divide $\ftt$; and $(ii)$ all finite semigroups divide $\ftt$.
\end{abstract} 


\maketitle

\section{Introduction}

Let $S$ be a semigroup. An \emph{inverse} of an element $s\in S$ is another element $s'\in S$ such that $ss's=s$ and $s'ss'=s'$. We denote by $V(s)$ the set of all inverses of $s$ in $S$. A semigroup where all elements have at least one inverse, that is, no set $V(s)$ is empty, is called \emph{regular}. It is well known that not all subsemigroups of regular semigroups are regular. Thus, for classes of regular semigroups, the usual concept of variety of algebras is not appropriate. Instead, the concept of e-variety is used. An \emph{e-variety} \cite{ha1,ks1} of regular semigroups is a class of these algebras closed for homomorphic images, direct products, and regular subsemigroups.

For e-varieties, there is a concept similar to the concept of `free object' for varieties. If $X$ is a nonempty set, let $X'=\{x'\,|\;x\in X\}$ be a disjoint copy of $X$, and set $\oX=X\cup X'$. A \emph{matched mapping} is a mapping $\phi:\oX\to S$ such that $x'\phi$ is an inverse of $x\phi$ in $S$, for all $x\in X$. If $\bf V$ is an e-variety of regular semigroups, a \emph{bifree object} in $\bf V$ on $X$ is a semigroup $BF{\bf V}(X)\in {\bf V}$, together with a matched mapping $\iota:\oX\to BF{\bf V}(X)$, such that any other matched mapping $\phi:\oX\to S$, with $S\in {\bf V}$, has a unique extension into a homomorphism $\varphi:BF{\bf V}(X)\to S$.

There is, however, an important difference between varieties and e-varie\-ties: not all e-varieties have bifree objects. In fact, Yeh \cite{yeh} proved that only the e-varieties of locally inverse semigroups and the e-varieties of regular $E$-solid semigroups have bifree objects on a set $X$ with at least two elements. Yeh showed that if $S$ is a locally inverse semigroup or a regular $E$-solid semigroup, and if $\phi:\oX\to S$ is a matched mapping, then $S$ has the smallest regular subsemigroup containing $\oX\phi$. Using this fact, Yeh was able to prove that all e-varieties of locally inverse semigroups and all e-varieties of regular $E$-solid semigroups have bifree objects on any set $X$. 

In addition, Yeh \cite{yeh} constructed a finite semigroup $S$, together with a matched mapping $\phi:\oX\to S$ with $|X|=2$, such that $S$ had no smallest regular subsemigroup containing $\oX\phi$. The semigroup $S$ had two regular subsemigroups, $S_1$ and $S_2$, both containing $\oX\phi$, such that $S_1\cap S_2$ was not regular. Then, Yeh showed that $S$ belongs to all e-varieties $\bf V$ containing non-locally inverse and non-regular $E$-solid semigroups, and concluded that $\bf V$ had no bifree objects on $X$ since, otherwise, $\phi$ would not have a unique extension. It was then obvious that $\bf V$ had also no bifree objects on any set with more than two elements.

In this paper, a regular semigroup $S$ with no proper regular subsemigroups containing a set $X$ is said to be \emph{weakly generated} by $X$. Note that this does not mean that $S$ is generated by $X$ itself since, very often, the subsemigroup of $S$ generated by $X$ is a proper non-regular subsemigroup. As demonstrated by the semigroup constructed by Yeh mentioned above, a regular semigroup can have several distinct regular subsemigroups weakly generated by the same set $X$. In fact, as described above, the existence of bifree objects on an e-variety $\bf V$ is closely related with the property that all semigroups $S\in{\bf V}$ have a unique regular subsemigroup weakly generated by $\oX\phi$, for any matched mapping $\phi:\oX\to S$.

Although regular semigroups weakly generated by a set seem to be important in the theory of regular semigroups, little research has been done about their structure. The goal of this paper is to contribute to the knowledge about the structure of those semigroups. However, we will focus only on the case of the regular semigroups weakly generated by idempotents. We took this option not only because it is a simpler situation to begin with, but also because regular semigroups weakly generated by a set $X$ of idempotents are, in fact, regular idempotent generated semigroups, although usually not generated by $X$.

The (regular) idempotent generated semigroups have been a topic of great interest in Semigroup Theory with a vast literature on the subject. One of the reasons is the existence of natural examples of these semigroups, such as the semigroup of all singular transformations on a finite set and the semigroup of all singular $n\times n$ matrices over a field. From the pure semigroup theoretical point of view, the interest in these semigroups comes also from both the facts that all (finite) semigroups can be embedded into (finite) idempotent generated semigroups and  that the set of idempotents of a semigroup carries a lot of information about the semigroup itself and its idempotent generated subsemigroup.

Nambooripad \cite{namb75} characterized the set of idempotents of regular semigroups abstractly as partial algebras, the so-called regular biordered sets, and proved \cite{namb79} the existence of a `free regular idempotent generated semigroup $RIG(E)$ on a regular biordered set $E$'. These results were then generalized by Easdown to all semigroups. Easdown \cite{ea85} characterized the set of idempotents of a semigroup abstractly as a biordered set and proved the existence of a `free idempotent generated semigroup $IG(E)$ on a biordered set $E$'. The semigroups $IG(E)$ and $RIG(E)$ have deserve a lot of attention since then.

The maximal subgroups of $IG(E)$ and $RIG(E)$ have been one of the main topics of research related to these semigroups. The reasons for this interest are explained in \cite{sz20}. The maximal subgroups of the earlier examples of $IG(E)$ and $RIG(E)$ were all free groups \cite{napa80, mc02}. This led to the conjecture that the maximal subgroups of both $IG(E)$ and $RIG(E)$ were always free. However, in 2009, Brittenham, Margolis and Meakin \cite{brmame09} gave a first example that disproves this conjecture. In fact, Gray and Ruskuc \cite{grru12} proved that quite the opposite occurs: every group is a maximal subgroup of some $IG(E)$ and of some $RIG(E)$. In the past decade much research has been carried on these semigroups related to their maximal subgroups (see, for example, \cite{dadogo19, dago16, dogr14, dogrru17, doru13, goya14, grru12b}). We refer the reader to the survey paper \cite{sz20} for further details on this topic. We just mention that the research presented here may contribute with new questions on this topic. We briefly address this idea in the last section of this paper.

The main achievement of this paper is a result showing that the class $\wg$ of all regular semigroups weakly generated by a set $X$ of idempotents has a regular semigroup $\ftx$ such that all other regular semigroups of $\wg$ are homomorphic images of $\ftx$. Thus $\ftx$ is a sort of ``free object'' in $\wg$. We will effectively construct the semigroup $\ftx$, and we will then study its structure. 

This paper is organized as follows. In the next section, we introduce the object of study, the regular semigroups weakly generated by a fixed set of idempotents, in a more formal setting. We also recall some basic concepts used in Semigroup Theory that we will need for the remainder of the paper. 

In Section 3 we introduce the semigroup $\ftx$. We begin by constructing its infinite set $\Gx$ of generators from $X$. Then, we define $\ftx$ as a quotient semigroup $\Gx^+/\rho$, where $\rho$ is the congruence generated by two infinite sets, $\rho_e$ and $\rho_s$, of relations. Thus, $\ftx$ is introduced as the semigroup given by a presentation $\langle X,\rho_e\cup\rho_s\rangle$, where both the set of generators and the set of generating relations are infinite. However, although both $\Gx$ and $\rho_e\cup\rho_s$ are infinite sets, we can show that each $\rho$-class has a canonical element, called a mountain, thus solving the word problem for this presentation. We end Section 3 by constructing a model $M(X)$ for $\ftx$, where the elements of $M(X)$ are precisely the mountains.

In Section 3 we prove also that $\ftx$ is a regular semigroup and that the elements of $\Gx$ are idempotents. But, to advance and conclude that $\ftx\in\wg$, we need to study the structure of $\ftx$ first. This study will be done at the beginning of Section 4. The remainder of Section 4 is devoted to proving that $\ftx\in\wg$. Finally, in Section 5, we prove that all semigroups from $\wg$ are homomorphic images of $\ftx$. However, we also observe that the converse is not true, that is, we construct an example of a homomorphic image of $\ftx$ that is not weakly generated by $X$. We end Section 5 with a brief reflection on decidability questions that can now be posed (and are still open) about regular semigroups weakly generated by a set of idempotents.

We will denote $\ftx$ by $\ftn$ if $|X|=n$. In Section 6, we compare the different $\ftn$ and prove, somehow surprisingly, that each $\ftn$ can be embedded into $\ftt$. This result show us how complex must be the structure of $\ftt$ since we can immediately deduce that all regular semigroups weakly generated by a finite set of idempotents strongly divide $\ftt$.

We return to the structure of $\ftx$ in Section 7. In that section, we compare the different $\Dc$-classes of $\ftx$ and analyze its idempotents. One of the goals is to get more information about the structure of the biordered set of idempotents of $\ftx$. However, as it will become clear, more research is needed on this topic since it is not immediate to identify if an element of $\ftx$ is an idempotent.

In the last section of this paper we make some considerations about possible avenues for future research related with the research presented here.

\section{Preliminaries}

Let $A$ be a nonempty subset of a semigroup $S$. A regular subsemigroup $T$ of $S$ containing $A$ is said to be weakly generated by $A$ if $T$ has no proper regular subsemigroups containing $A$. Note that, very often, $T$ is not generated by $A$ as a semigroup. In fact, as mentioned earlier, $S$ can have several (or none) distinct regular subsemigroups weakly generated by the same subset $A$. If $S$ is regular and has no proper regular subsemigroups containing $A$, then we say that $S$ is weakly generated by $A$. 

More generally, if $X$ is a formal nonempty set and $S$ is an arbitrary regular semigroup, we shall say that $S$ is weakly generated by $X$ if there exists a one-to-one mapping $\phi:X\to S$ such that $S$ is weakly generated by the subset $X\phi$. We shall denote by $\wgen(X)$ the category of all regular semigroups weakly generated by $X$, that is,
\begin{itemize}
\item[$(i)$] the objects of $\wgen(X)$ are the pairs $(S,\phi:X\to S)$, where $S$ is a regular semigroup and $\phi:X\to S$ is a one-to-one mapping such that $S$ is weakly generated by the subset $X\phi$; and
\item[$(ii)$] the morphisms from $(S_1,\phi_1:X\to S_1)$ to $(S_2,\phi_2:X\to S_2)$ are the usual semigroup homomorphisms $\varphi:S_1\to S_2$ that verify also $\phi_1\varphi=\phi_2$.
\end{itemize}
Given an object $(S,\phi:X\to S)\in\wgen (X)$, we shall identify the set $X$ with its image under $\phi$. In this manner, we avoid having to refer to the mapping $\phi$, and we shall write only $S\in\wgen(X)$. Note that, in particular, given $S,T\in\wgen(X)$, a morphism $\varphi:S\to T$ is a usual semigroup homomorphism such that $\varphi_{|X}$ is the `identity' mapping.

In this paper, we will be working simultaneously in the category $\wgen(X)$ and in the usual category of all semigroups. So, we must be careful and have a way to clearly distinguish the two cases. If we say that $\varphi:S\to T$ is a morphism, then we are implicitly assuming that $S,T\in\wgen(X)$, for some nonempty set $X$, and that $\varphi$ is a morphism in the category $\wgen(X)$. Thus, if we say that $T$ is a morphic image of $S$, then we are again implicitly assuming that $S,T\in\wgen(X)$, for some nonempty set $X$, and that there exists a surjective morphism $\varphi:S\to T$ in $\wgen(X)$. On the other hand, if we say that $\varphi:S\to T$ is a homomorphism, then this means that we are working in the usual category of all semigroups and that $\varphi$ is just a usual semigroup homomorphism. Thus, if we say that $T$ is a homomorphic image of $S$, then we are just saying that there exists a surjective semigroup homomorphism $\varphi:S\to T$. 

To clarify the usage of this terminology, if we know that $S\in\wgen(X)$, that $X$ is a subset of $T$, and that $\varphi:S\to T$ is a semigroup homomorphism such that $\varphi_{|X}$ is the identity mapping, but we are not sure that $T\in\wgen(X)$, then we shall not call $\varphi$ a morphism; in this case, we shall write that $\varphi:S\to T$ is a homomorphism such that $\varphi_{|X}$ is the identity mapping. Summing up, we use the terms morphism and morphic image if we know that we are working in $\wgen(X)$, for some nonempty set $X$; otherwise, we use the terms homomorphism and homomorphic image.

The goal of this paper is to study the structure of the objects from a particular subcategory of $\wgen(X)$, namely
$$\wg=\{(S,\phi:X\to S)\in\wgen(X)\,|\; X\phi\subseteq E(S)\}\,,$$
where $E(S)$ denotes the set of idempotents of $S$, as usual. As above, we continue to identify $X$ with its image under $\phi$. In this case, $X$ will be always a set of idempotents of $S$. The semigroups from $\wg$ will be said to be weakly idempotent generated by $X$. We remark that a regular semigroup $S$ weakly idempotent generated by $X$ is always an idempotent generated regular semigroup. In fact, since $\langle E(S)\rangle$ is always a regular subsemigroup of $S$ (see \cite{fitz}) and it constains $X$, we must have $S=\langle E(S)\rangle$.

Given a formal nonempty set $X$, we denote by $X^+$ and $X^*$ the free semigroup and the free monoid on $X$, respectively, as usual. The elements of $X$ are called \emph{letters}, while the elements of $X^*$ are called \emph{words}. The \emph{content} $\c(u)$ of a word $u$ is the set of all letters from $X$ that occur in $u$. Usually, one denotes the empty word of $X^*$ by $1$. Given a word $u\in X^+$, $\si(u)$ and $\tau(u)$ will denote the first and the last letter of $u$, respectively. The \emph{length} $\l(u)$ of a word $u$ is the number of letters that occur in $u$. Thus $1$ is the only word of length $0$. If $\l(u)=n$ and $i\leq n$, then $\si_i(u)$ denotes the prefix of $u$ of length $i$. Similarly, $\tau_i(u)$ denotes the suffix of $u$ of length $i$. We will also need the notion of reverse of a word: if $u=x_1\cdots x_n$ with $x_i\in X$ for $i=1,\cdots,n$, then its \emph{reverse}  is the word $\cev{u}=x_n\cdots x_1$.

Let $S$ be a semigroup. The notions of left principal ideal, right principal ideal, and principal ideal induce the well-known Green relations: for any $s,t\in S$,
$$s\Lc t\Leftrightarrow S^1s= S^1t,\quad s\Rc t\Leftrightarrow sS^1= tS^1,\quad s\Jc t\Leftrightarrow S^1sS^1= S^1tS^1,$$
$$\Hc=\Lc\cap\Rc\qquad\mbox{ and }\qquad \Dc=\Lc\vee\Rc\,,$$
where $S^1$ is the monoid obtained from $S$ by adding an identity element if necessary. They induce also three quasi-orders:
$$s\leq_{\Lc} t\Leftrightarrow S^1s\subseteq S^1t,\quad s\leq_{\Rc} t\Leftrightarrow sS^1\subseteq tS^1,\quad s\leq_{\Jc} t\Leftrightarrow S^1sS^1\subseteq S^1tS^1.$$
We denote by $\Kcc_a$ the $\Kc$-class of the element $a$ for $\Kc\in\{\Hc,\Lc,\Rc,\Dc,\Jc\}$. 

If $S$ is a regular semigroup, then we can replace $S^1$ with $S$ in the definitions of the previous relations. There is also another important relation on $S$ for the cases where $S$ is regular, the \emph{natural partial order} $\leq\,$:
$$s\leq t\quad\Leftrightarrow\quad  s=et=tf \;\mbox{ for some } e,f\in E(S)\,.$$
It is well-known that, in the previous definition of the natural partial order, one can choose idempotents $e$ and $f$ such that $e\Rc s\Lc f$. We will use this fact later on this paper.

Given two idempotents $e$ and $f$ of $S$, the \emph{sandwich set} of $e$ and $f$ is the set
$$S(e,f)=V(ef)\cap E(fSe)\,.$$
There are other equivalent definitions for the sandwich set. One of then is $S(e,f)=fV(ef)e$. Another equivalent definition is 
$$S(e,f)=\{g\in E(S)\,|\; fg=g=ge \mbox{ and } egf=ef\}\,.$$
Note that $S(e,f)$ is always nonempty if $S$ is regular. In fact, $S(e,f)$ is always a rectangular band, and thus a subsemigroup, whenever it is nonempty. These sandwich sets have other interesting properties. For example, if $e,e_1,f,f_1\in E(S)$ are such that $e\Lc e_1$ and $f\Rc f_1$, then $S(e,f)=S(e_1,f_1)$. Hence, for regular semigroups, one can extend the definition of sandwich set to all elements of $S$ as follows: for $a,b\in S$, let $S(a,b)=S(a'a,bb')$ for some (any) $a'\in V(a)$ and $b'\in V(b)$.

\section{A idempotent generated regular semigroup $\ftx$}

The goal of this paper is to introduce, for each nonempty set $X$, a semigroup $\ftx$ and to show that $\ftx$ plays an important universal role in the category $\wg$. In this section, we introduce such semigroup. We begin by defining its set of generators $\Gx$. Then, $\ftx$ will be defined as a quotient of the free semigroup $\Gx^+$. After, we show there exists a canonical form for the elements of $\ftx$, and thus solving its word problem. This canonical form will allow us to prove that $\ftx$ is an idempotent generated regular semigroup. We end this section describing a model for $\ftx$.

Throughout this paper we will often need to refer to the entries of triples $g\in L\times C\times R$, where $L$, $C$ and $R$ are nonempty sets. For that purpose, we shall use the following notation without further comments: $g^l$, $g^c$ and $g^r$ refer to the left, middle and right entries of $g$, respectively. In other words, $g=(g^l,g^c,g^r)$.

Next, we define recursively a sequence $(\Gix)_{i\in\mathbb{N}_0}$ of nonempty sets $\Gix$. We start by setting $\G_0(X)=\{1\}$ and $\G_1(X)=X$, where $1$ is a new symbol not in $X$. For the definition of $\G_2(X)$ to make sense, we need also to agree on identifying each element $x\in X$ with the triple $(1,x,1)$, so that we can write $x^l=x^r=1$. Now, for $i\geq 2$, let 
$$\overline{\G}_i(X)=\G_{i-1}(X)\times \G_{i-2}(X)\times \G_{i-1}(X)$$
and
$$\Gix=\left\{g\in \overline{\G}_i(X)\,|\; g^l\neq g^r,\; g^c\in\{(g^l)^l,(g^l)^r\}\cap\{(g^r)^l,(g^r)^r\}\right\}\,.$$ 
Note that if $g\in \Gix$ for some $i\geq 2$, then also $g_1=(g^r,g^c,g^l)$ belongs to $\Gix$. Further, $g\neq g_1$ since $g^l\neq g^r$, and so 
$$\{(g,g^l,g_1),(g,g^r,g_1),(g_1,g^l,g),(g_1,g^r,g)\}\subseteq \G_{i+1}(X)\,.$$ 
Therefore, all $\Gix$ are nonempty sets.  

As a title of example, if $X=\{e,f\}$, then $\G_1(X)=X$, $\G_2(X)=\{e_1,f_1\}$ for 
$$e_1=(e,1,f)\quad\mbox{ and }\quad f_1=(f,1,e)\,,$$ 
and $\G_3(X)=\{h,h_1,h_2,h_3\}$\label{g123} for
$$h=(e_1,f,f_1),\;\; h_1=(e_1,e,f_1),\;\; h_2=(f_1,f,e_1)\;\;\mbox{ and }\;\; h_3=(f_1,e,e_1)\,.$$ 
The reader should keep these three sets in mind since they will be used in the examples of Section \ref{sec5}.

Let $\Gx=\cup_{i\in\mathbb{N}_0}\Gix$. The \emph{height} of $g\in \Gx$ is the index $\up(g)=i\in\mathbb{N}_0$ such that $g\in \Gix$. To work with the elements of $\Gx$, we need to work with triples nested inside triples, which themselves are nested inside other triples, and so on. Thus, we need to introduce some more notation to deal with the nested triples of $\Gx$. We begin by generalizing the notations $g^l$, $g^c$ and $g^r$. Let $u=x_1\cdots x_n$ be a nonempty word on the alphabet $\{c,l,r\}$, that is, $x_i\in\{c,l,r\}$ for $1\leq i\leq n$. Then $g^u=(((g^{x_1})^{x_2}\cdots )^{x_n}$. We define also $g^1=g$.

Note that, for $g\in \Gix$ with $i\geq 2$, $g^c$ is either the left or right entry of $g^l$. It will be useful to have a generic way for referring to the side of the entry $g^c$ inside $g^l$. So, let $g^{l_c}$ denote the entry $g^c$ inside $g^l$, that is, $l_c=l^2$ if $g^c=g^{l^2}$, or $l_c=lr$ if $g^c=g^{lr}$. Therefore, $g^c$ and $g^{l_c}$ represent the same element of $\Gx$, but they have different meanings: $g^c$ is the central entry of $g$, while $g^{l_c}$ is either the left entry or the right entry of $g^l$. Usually, there are also other elements $g_1\in \Gix$ with $g_1^l=g^l$, but $g_1^{l_c}\neq g^{l_c}$. This means that we cannot compute $g^{l_c}$ directly from $g^l$. In other words, we must always know $g$ for computing the entry $g^{l_c}$ of $g^l$.

Let $g^{l_o}$ denote the other entry of $g^l$ distinct from both $g^{l_c}$ and $g^{lc}$. Thus $g^{l_c},g^{l_o}\in \G_{i-2}(X)$ if $g\in \Gix$. As expected, we will use also the dual notations $g^{r_c}$ and $g^{r_o}$. Hence $g^{r_c}=g^c=g^{l_c}$ and $g^{r_o}\neq g^c\neq g^{l_o}$. Note, however, that $g^{r_o}$ and $g^{l_o}$ may be the same element of $\Gx$ (see, for example, the element $h$ of the set $\G_3(X)$ above). Now, if $u=x_1\cdots x_n$ with $x_i\in\{c,l,r,l_c,l_o,r_c,r_o\}$ for $1\leq i\leq n$, then let $g^u=(((g^{x_1})^{x_2})\cdots)^{x_n}$. For example, if $u=rl_oc$, then $g^u=((g^r)^{l_o})^c$.

Consider the free semigroup $\Gxp$ on $\Gx$. In this paper, if nothing is said in contrary, we will use $g$ and $h$ (with possible indices) to refer to elements of $\Gx$, and $u$ and $v$ (also with possible indices) to refer to words from $\Gxp$. There are special characteristics that some words from $\Gxp$ have that will be important to us. Next, we list some general terms and notations that we will use throughout this paper. 

\begin{description}
\item[Landscape] word $u=g_0\cdots g_n\in \Gxp$ ($n\geq 0$) such that either $g_{i-1}\in\{g_i^l,g_i^r\}$ or $g_i\in\{g_{i-1}^l,g_{i-1}^r\}$ for all $1\leq i\leq n$. We denote by $\lx$ the set of all landscapes of $\Gxp$. Note also that single letters $g\in \Gx$ are particular landscapes; whence $\Gx\subseteq \lx$. 
\item[Ridge] letter $g_i$ of a landscape such that $\up(g_{i-1})=\up(g_{i+1})=\up(g_i)-1$. Hence, a ridge is never the first nor the last letter of a landscape.
\item[Peak] highest ridge of a landscape $u$. Note that $u$ can have several peaks, but all have the same height. If $u$ has only one peak, we denote it by $\ka(u)$. The height $\up(u)$ of the landscape $u$ is the maximum between the height of its peaks, the height of $\si(u)$, and the height of $\tau(u)$. Note that this definition of height agrees with the notion of height for elements of $\Gx$ introduced earlier.
\item[River] letter $g_i$ of a landscape such that $\up(g_{i-1})=\up(g_{i+1})=\up(g_i)+1$. As with ridges, rivers are never endpoints of landscapes.
\item[Hill] $u=g_0\cdots g_n\in \lx$ with $n\geq 1$ such that either $g_{i-1}\in\{g_i^l,g_i^r\}$ for all $1\leq i\leq n$, or $g_i\in\{g_{i-1}^l,g_{i-1}^r\}$ for all $1\leq i\leq n$. In the former case we have an \emph{uphill} since $\up(g_i)=\up(g_{i-1})+1$, while on the latter case we have a \emph{downhill} since $\up(g_{i})=\up(g_{i-1})-1$. We denote by $\lo$, $\lop$ and $\lom$ the set of all hills, all uphills and all downhills, respectively. Further, $\ltp$ will denote the set of all landscapes composed by an uphill followed by a downhill.
\item[Valley] landscape composed by a downhill followed by an uphill. Thus, a valley has always one (and only one) river at its lower height letter. We denote by $\ltm$ the set of all valleys of $\Gxp$, and by $\lt$ the set $\ltp\cup\ltm$. 
\item[Canyon] valley $u$ with $\si(u)=\tau(u)$. We denote by $\cx$ the set of all canyons of $\Gxp$.
\item[Mountain range] landscape $u$ with $\si(u)=\tau(u)=1$. Thus, the height of a nontrivial mountain range is the height of its peaks. Note that all mountain ranges have odd length, and $u=1$ is the only mountain range with length $1$. We denote by $\mrx$ the set of all mountain ranges of $\Gxp$.
\item[Mountain] mountain range with no rivers. Thus, it is either the trivial mountain range $u=1$, or it is composed by an uphill followed by a downhill. Nontrivial mountains have only one ridge. We denote by $\mxm$ and $\mx$ the sets of all mountains and all nontrivial mountains, respectively. So, $\mx=\mrx\cap \ltp$.
\end{description}

We shall use the previous terminology also to refer to subwords of a given word $u\in \Gxp$. For example, a hill of $u\in\lx$ is a subword of $u$ that is also a hill. A hill of $u\in\lx$ is called maximal if it is not properly contained in another hill of $u$, while a valley of $u$ is called maximal if it is not properly contained in another valley of $u$. Thus, each nontrivial mountain range is uniquely decomposable into an initial maximal uphill, followed by a possible empty sequence of maximal valleys between its consecutive ridges, and ending with a maximal downhill. 
However, we must be careful since $u_1u_2$ is not a landscape if $u_1$ and $u_2$ are two landscapes such that $g=\tau(u_1)=\si(u_2)$: it appears a double $gg$ at the junction of $u_1$ with $u_2$. Let $u_1*u_2$ denote the landscape obtained from $u_1u_2$ by replacing $gg$ with $g$.

We associate a mountain to each $g\in \Gx\setminus\{1\}$ as follows:
$$\be_1(g)=\left\{\begin{array}{ll}
\!\!g^{c^n}g^{c^{n-1}l}g^{c^{n-1}}\cdots g^cg^lgg^rg^c\cdots g^{c^{n-1}}g^{c^{n-1}r}g^{c^n} \!&\mbox{if } \up(g)=2n \\ [.2cm]
\!\!1g^{c^n}g^{c^{n-1}l}g^{c^{n-1}}\!\cdots g^cg^lgg^rg^c\cdots g^{c^{n-1}}\!g^{c^{n-1}r}g^{c^n}1\!\! &\mbox{if } \up(g)=2n+1.\end{array}\right.$$
Thus, $\be_1(g)$ is a mountain composed by an uphill $\la_l(g)$, called its \emph{left hill}, followed by a downhill $\la_r(g)$, called its \emph{right hill}. Hence $\be_1(g)=\la_l(g)*\la_r(g)$. For technical reasons, we define also $\be_1(1)=\la_l(1)=\la_r(1)=1$.

Let $\rho$ be the smallest congruence on $\Gxp$ containing both sets
$$\rho_e=\{(1g,g),(g1,g),(g^2,g)\,|\;g\in \Gx\}$$
and
$$\rho_s=\left\{(g^cg^lg,g),(gg^rg^c,g),(g^rg^cgg^cg^l,g^r g^cg^l)\,|\; g\in \Gx \mbox{ and }\up(g)\geq 2\right\}.$$
We denote by $\ftxm$ the semigroup $\Gxp/\rho$. We can immediately see that $\rho_e$ turns $\ftxm$ into an idempotent generated monoid with identity element $1\rho$. Further, the class $1\rho$ is constituted by all words with content $\{1\}$. Let $ \Gx^{\oplus}$ be the subsemigroup of all words form $\Gxp$ with content different from $\{1\}$, and let $\ftx$ be the semigroup $\ftxm\setminus\{1\rho\}$. Then $\ftx= \Gx^{\oplus}/\rho$. We may write $\ftnm$ and $\ftn$ instead of $\ftxm$ and $\ftx$, respectively, if $|X|=n$. 

To avoid using a heavier notation, we will try not to write the congruence $\rho$ in our expressions. Thus, we will represent the $\rho$-class of $u$ by $[u]$. If $A$ is a subset of $\Gx^+$, then 
$$[A]=\{[u]\,|\; u\in A\}\,.$$ 
We will write also $u\ap v$ to indicate that $[u]=[v]$ in $\ftxm$. If we write $u=v$, we are effectively saying that $u$ and $v$ are the same word. For example, for $g\in \Gx$, $\be_1(g)\ap \la_l(g)\la_r(g)$ but $\be_1(g)\neq \la_l(g)\la_r(g)$ because of the double $gg$ that occurs at the junction of $\la_l(g)$ with $\la_r(g)$ in $\la_l(g)\la_r(g)$. 

\begin{lem}\label{prodelem}
Let $g,g_1\in \Gx$ with $\up(g)>1$. Then:
\begin{itemize}
\item[$(i)$] $g\ap g^cg\ap gg^c\ap gg^lg\ap gg^rg$ and $g^rgg^l\ap g^rg^cg^l$.
\item[$(ii)$] $[gg^r],\,[gg^l],\,[g^rg],\,[g^lg]\in E(\ftx)$.
\item[$(iii)$] $[g^lg]\Lc [g^rg]\Lc [g]\Rc [gg^r]\Rc [gg^l]$.
\item[$(iv)$] $[g]\in S([g^rg^c],[g^cg^l])$.
\item[$(v)$] $g_1\ap \be_1(g_1)$.
\end{itemize}
\end{lem}

\begin{proof}
Note that $gg^c\ap gg^rg^cg^c\ap gg^rg^c\ap g$ and $gg^lg\ap gg^cg^lg\ap g^2\ap g$ by definition of $\rho_e$ and $\rho_s$. We prove that $g^cg\ap g$ and $gg^rg\ap g$ similarly. Hence, we have also $g^rgg^l\ap g^rg^cgg^cg^l\ap g^rg^cg^l$, and we have finished the proof of $(i)$. Now, $(ii)$ and $(iii)$ follow from $gg^lg\ap g\ap gg^rg$, while $(iv)$ is just the definition of $\rho_s$ and of the sandwich set since both $[g^rg^c]$ and $[g^cg^l]$ are idempotents by $(ii)$. From the definition of $\rho$ we also conclude that 
$$g^{c^i}g^{c^{i-1}l}g^{c^{i-1}}\ap g^{c^{i-1}}\quad\mbox{ and }\quad g^{c^{i-1}}g^{c^{i-1}r}g^{c^i}\ap g^{c^{i-1}}$$
for all $i$ such that $0<2i\leq\up(g)$. Hence $\be_1(g)\ap g$. Clearly also $\be_1(g_1)\ap g_1$ if $\up(g_1)\leq 1$.
\end{proof}

Let $u=g_0\cdots g_n\in \Gxp$. Extend $\be_1$ to $u$ as follows:
$$\be_1(u)=\be_1(g_0)*\cdots*\be_1(g_n)\,.$$ 
Note that $\be_1(u)\in \mrx$. The following result is obvious from Lemma \ref{prodelem}.$(v)$.

\begin{lem}\label{m1}
$\be_1(u)\ap u$ for all $u\in \Gxp$.
\end{lem}

Let $u=g_0\cdots g_n\in \lx$ and assume that $g_i$ is a river of $u$. Set
$$v=\left\{\begin{array}{ll}
g_0\cdots g_{i-1}g_{i+2}\cdots g_n &\mbox{ if } g_{i-1}=g_{i+1}\,, \mbox{ or}\\ [.2cm]
g_0\cdots g_{i-1}g_i'g_{i+1}\cdots g_n\quad &\mbox{ if } g_{i-1}\neq g_{i+1}\,,
\end{array}\right.$$
where $g_i'=(g_{i+1},g_i, g_{i-1})\in \Gx$. Then $v$ is another landscape with length not greater than the length of $u$. Furthermore, $v\in \mrx$ if $u\in \mrx$. We say that $v$ is obtained from $u$ by \emph{uplifting a river} and write $u\to v$ (or $u\xrightarrow{g_i} v$ if one needs to identify the river uplifted). If $g_{i-1}\neq g_{i+1}$, then the uplifting of $g_i$ replaces the river $g_i$ with a ridge $g_i'$. Furthermore, $g_{i-1}$ and $g_{i+1}$ become rivers of $v$ unless they were ridges of $u$, respectively. If $g_{i-1}=g_{i+1}$, then the uplifting of $g_i$ eliminates the river $g_i$, but creates a new river $g_{i-1}$ if neither $g_{i-1}$ nor $g_{i+1}$ are ridges of $u$; otherwise, it just eliminates the river $g_i$ and creates no new ridge nor river. We use the symbol $\xrightarrow{*}$ to denote the reflexive and transitive closure of $\to$.

Let $u$ be a landscape of length $n$ and height $i$, and consider the set $U(u)=\{v\in\Gxp\,|\;u\xr{*}v\}$. Note that if $v\in U(u)$, then $v$ is landscape with length at most $n$, height at most $j=\lfloor i+n/2\rfloor$, and such that $\si(v)=\si(u)$ and $\tau(v)=\tau(u)$. Further, $v$ has less than $\lceil n/2\rceil$ rivers. Hence, we can record the number of rivers of each height of $v$ in a $j$-tuple $r(v)=(r^1(v),\cdots,r^j(v))$, where $r^k(v)$ is the number of rivers of $v$ of height $k-1$. Thus, the sum of all entries of $r(v)$ is less than $\lceil n/2\rceil$. Let $R(u)=\{r(v)\,|\; v\in U(u)\}$. Then $R(u)$ is a finite set and we can consider it ordered by the lexicographic order. Hence, the $j$-tuple with all entries equal to $0$ is the smallest $j$-tuple of $R(u)$, and it corresponds to the landscapes of $U(u)$ with no rivers, that is, the landscapes of $U(u)$ that belong also to $\ltp\cup\lo\cup \Gx$.

If $v\in U(u)$ and $v\to v_1$ by uplifting a river of height $k$, then $v_1\in U(u)$, $r^k(v_1)=r^k(v)-1$, $r^{k+1}(v)\leq r^{k+1}(v_1)\leq r^{k+1}(v)+2$, and $r^{k_1}(v_1)=r^{k_1}(v)$ for all other $k_1$. So $r(v_1)<r(v)$. Consequently, we cannot apply upliftings of rivers indefinitely to $u$, and we must always stop after a finite number of upliftings of rivers with a landscape with no rivers, that is, a landscape from $\ltp\cup\lo\cup \Gx$.

It is also easy to see that the uplifting of rivers is a commutative operation. In other words, if $g_i$ and $g_j$ are two rivers of $u\in\lx$, then the uplifting of $g_i$ followed by the uplifting of $g_j$ gives the same landscape as the uplifting of $g_j$ followed by the uplifting of $g_i$. Hence, the uplifting of rivers is a system of rules commonly known as a noetherian locally confluent system of rules. An important property of this kind of systems is that, independently of the order of the rules that we apply to an element, we must always stop after a finite number of steps with the same `reduced' element (see \cite{newman}). In the case considered here, this means that, independently of the order of upliftings of rivers that we apply to $u\in\lx$, we will always end up with the same landscape from $U(u)$ with no rivers. In particular, $U(u)$ has a unique landscape with no rivers. We designate it by $\be_2(u)$. Note further that $\be_2(u)=\be_2(v)$ if $u\in \lx$ and $u\xr{*} v$, and that $\be_2(u)\in\mxm$ if $u\in\mrx$.

\begin{lem}\label{m2}
$\be_2(u)\ap u$ for all $u\in \lx$.
\end{lem}

\begin{proof}
Let $g_i$ be a river of $u$ and let $v$ be obtained from $u$ by uplifting $g_i$. By Lemma \ref{prodelem}.$(i)$,
$$g_{i-1}g_ig_{i+1}\ap\left\{\begin{array}{ll}
g_{i-1} & \mbox{ if } g_{i-1}=g_{i+1} \\ [.2cm]
g_{i-1}g_i'g_{i+1} \quad & \mbox{ if } g_{i-1}\neq g_{i+1}\,,
\end{array}\right.$$
where $g_i'=(g_{i+1},g_i,g_{i-1})\in \Gx$. Hence $u\ap v$. Now, applying several times the previous conclusion, we obtain $\be_2(u)\ap u$.
\end{proof}

For each $u\in \Gxp$, let $\be(u)=\be_2(\be_1(u))$. Then $\be(u)\in \mxm$ and, by the two previous lemmas, $\be(u)\ap u$. We register this observation in the next corollary for future reference. 

\begin{cor}\label{m}
$\be(u)\ap u$ for all $u\in \Gxp$.
\end{cor}

The next result is a consequence of Lemma \ref{prodelem}.$(i)$.

\begin{lem}\label{updownhill}
If $u=g_0\cdots g_n\in \lom$, then $u\cev{u}\ap g_0$.
\end{lem}

\begin{proof}
Note that $g_i\in\{g_{i-1}^l,g_{i-1}^r\}$, for $i=1,\cdots, n$, since $u\in\lom$. Now, applying Lemma \ref{prodelem}.$(i)$ several times, we have
$$\begin{array}{ll}
u\cev{u} \ap u*\cev{u}\!\!\! & = g_0\cdots g_{n-2}g_{n-1}g_ng_{n-1}g_{n-2}\cdots g_0 \\ [.2cm]
& \ap g_0\cdots g_{n-2}g_{n-1}g_{n-2}\cdots g_0 \\ [.2cm]
& \ap g_0\cdots g_{n-3}g_{n-2}g_{n-3}\cdots g_0 \\
&\hspace*{1cm} \vdots \\
& \ap g_0g_1g_0 \ap g_0\,.
\end{array}$$
So $u\cev{u}\ap g_0$.
\end{proof}

We can now prove that $\ftxm$ is a regular monoid.

\begin{prop}\label{reg}
The monoid $\ftxm$ is an idempotent generated regular monoid and $[\cev{u}]$ is an inverse of $[u]$ for all $u\in \ltp$.
\end{prop}

\begin{proof}
We have already seen that $\ftxm$ is an idempotent generated mo\-noid. Since $v\ap \be(v)$ and $\be(v)\in \mxm\subseteq \ltp\cup\{1\}$ for all $v\in \Gxp$, it is enough to prove the second part of this proposition to conclude also that $\ftxm$ is regular. Let $u=g_0\cdots g_n\in \ltp$ with peak $g_k$ for some $0<k<n$. Then $u_1=g_0\cdots g_k\in\lop$, while $u_2=g_k\cdots g_n\in\lom$. By Lemma \ref{updownhill}, $u_2\cev{u_2}\ap g_k\ap\cev{u_1}u_1$ (note that $\cev{u_1}=g_k\cdots g_0\in\lom$). Hence
$$u\cev{u}u=u_1*u_2\,\cev{u_2}*\cev{u_1}\,u_1*u_2\ap u_1*g_k*g_k*u_2=u_1*u_2=u\,.$$
Since $\cev{v}=u$ for $v=\cev{u}$, we have similarly that $\cev{u}u\cev{u}\ap \cev{u}$. Consequently, $[\cev{u}]$ is an inverse of $[u]$.
\end{proof}

Now that we have shown that $\ftxm$ is a regular mo\-noid, let us prove that each $\rho$-class $\varrho$ has a canonical element. The next couple of results contain technical details that lead to the conclusion that $\be(u)=\be(v)$ if and only if $u\ap v$. Thus $\varrho$ contains exactly one mountain, namely $\be(u)$ for some (any) $u\in\varrho$. We will call $\be(u)$ the \emph{canonical form} of $u\in \Gxp$.

\begin{lem}\label{valelem}
Let $g\in \Gx$ and let $u=g_0\cdots g_n\in \lx$. Then: 
\begin{itemize}
\item[$(i)$] $\la_r(g)*\la_l(g)\xrightarrow{*}g$ and $\be_1(g^2)\xr{*} \be_1(g)$.
\item[$(ii)$] If $\up(g)\geq 2$ and $h\in \{g^r,g^l\}$, then $\la_r(h)*\la_l(g)\xrightarrow{*}hg$ and $\la_r(g)*\la_l(h)\xr{*}gh$.
\item[$(iii)$] $\be_1(u)\xr{*} \la_l(g_0)*u*\la_r(g_n)$.
\item[$(iv)$] $\be(u)=\be_2(u)$ if $u\in \mrx$.
\item[$(v)$] $\be(u)=u$ if $u\in \mxm$.
\end{itemize}
\end{lem}

\begin{proof}
$(i)$. We prove the first part of $(i)$ by induction on $\up(g)$. Clearly $\la_r(1)*\la_l(1)=1$ and $\la_r(g)*\la_l(g)=g1g\xrightarrow{1} g$ for $g\in \G_1(X)=X$. Let $g\in \Gix$ for $i\geq 2$ and assume that $\la_r(h)*\la_l(h)\xrightarrow{*}h$ for all $h\in \G_j(X)$ with $j<i$. Then
$$\la_r(g)*\la_l(g)=gg^r(\la_r(g^c)*\la_l(g^c))g^lg\xr{*}gg^rg^cg^lg\xr{g^c} gg^rgg^lg\xr{*} g$$
as desired. Now, the second part of $(i)$ follows easily from the first:
$$\be_1(g^2)=\la_l(g)*\la_r(g)*\la_l(g)*\la_r(g)\xr{*}\la_l(g)*g*\la_r(g)=\be_1(g)\,.$$

$(ii)$. We prove $(ii)$ also by induction on $\up(g)$. If $\up(g)=2$, then 
$$\la_r(h)*\la_l(g)=h1g^lg$$
and either $h=g^l$ or $h=g^r$. If $h=g^l$, then $h1g^lg\xr{1}hg$. If $h=g^r$, then $h1g^lg\xr{1} hgg^lg\xr{g^l}hg$. Thus, we get $\la_r(h)*\la_l(g)\xrightarrow{*}hg$ in both cases. Similarly, we show that $\la_r(g)*\la_l(h)\xrightarrow{*}gh$ if $\up(g)=2$.

Let $g\in \Gix$ for $i>2$ and let $h\in\{g^l,g^r\}$. Assume that $\la_r(h_1)*\la_l(h_2)\xrightarrow{*}h_1h_2$ and $\la_r(h_2)*\la_l(h_1)\xrightarrow{*}h_2h_1$ for all $h_2\in \G_j(X)$ such that $2\leq j<i$ and $h_1\in\{h_2^l,h_2^r\}$. Note that $g^c\in\{h^l,h^r\}$. Hence, 
$$\la_r(h)*\la_l(g)=\la_r(h)*\la_l(g^c)g^lg\xr{*}hg^cg^lg\xr{*}hg\,,$$
where the first $\xr{*}$ follows from the induction hypothesis, while the second $\xr{*}$ follows from the same arguments used in the case $\up(g)=2$. We can show that $\la_r(g)*\la_l(h)\xrightarrow{*}gh$ similarly.

$(iii)$. We show that $\be_1(u_i)\xr{*} \la_l(g_0)*u_i*\la_r(g_i)$, for $u_i=g_0\cdots g_i$, by induction on $i$. Clearly $\be_1(u_0)=\la_l(g_0)*u_0*\la_r(g_0)$. Assume that $\be_1(u_{i-1})\xr{*} \la_l(g_0)*u_{i-1}*\la_r(g_{i-1})$. Then, by the induction hypothesis and by $(ii)$,
$$\begin{array}{ll}
\be_1(u_i)& =\be_1(u_{i-1})*\be_1(g_i)\\ [.2cm]
& \xr{*} \la_l(g_0)*u_{i-1}*\la_r(g_{i-1})*\la_l(g_i)*\la_r(g_i) \\ [.2cm]
& \xr{*} \la_l(g_0)*u_{i-1}*(g_{i-1}g_i)*\la_r(g_i)=\la_l(g_0)*u_{i}*\la_r(g_i)\,.
\end{array}$$
We have proved by induction that $\be_1(u)\xr{*} \la_l(g_0)*u*\la_r(g_n)$.

$(iv)$ follows easily from $(iii)$ since $g_0=1=g_n$ if $u\in\mrx$:
$$\be(u)=\be_2(\be_1(u))=\be_2(\la_l(g_0)*u*\la_r(g_n))=\be_2(u)\,;$$
and $(v)$ is obvious from $(iv)$.
\end{proof}

\begin{prop}\label{mrhoinv}
Let $g\in \Gx$ and $u,v\in \Gxp$.
\begin{itemize}
\item[$(i)$] $\be(g)=\be(g1)=\be(1g)=\be(g^2)$.
\item[$(ii)$] If $\up(g)\geq 2$, then $\be(g)=\be(g^cg^lg)=\be(gg^rg^c)$ and $\be(g^rg^cgg^cg^l)=\be(g^rg^cg^l)$.
\item[$(iii)$] $\be(u)=\be(v)$ if and only if $u\ap v$.
\end{itemize}
\end{prop}

\begin{proof}
$(i)$. $\be(g)=\be(g1)=\be(1g)$ are obvious equalities. Further, by Lemma \ref{valelem}.$(i)$, $\be(g^2)=\be_2(\be_1(g^2))=\be_2(\be_1(g))=\be(g)$ and $(i)$ is proved.

$(ii)$. By Lemma \ref{valelem}.$(iii)$, $\be_1(g^cg^lg) \xr{*} \la_l(g^c)*(g^cg^lg)*\la_r(g)=\be_1(g)$. Thus $\be(g^cg^lg)=\be(g)$. We have also $\be(gg^rg^c)=\be(g)$ by duality. Once more by Lemma \ref{valelem}.$(iii)$, we have 
$$\be_1(g^rg^c)\xr{*}\la_l(g^r)*(g^rg^c)*\la_r(g^c)=\la_l(g^r)\la_r(g^c)$$ 
and $\be_1(g^cg^l)\xr{*}\la_l(g^c)\la_r(g^l)$. Hence,
$$\begin{array}{ll}
\be_1(g^rg^cgg^cg^l)& \xr{*} \la_l(g^r)\la_r(g^c)*\be_1(g)*\la_l(g^c)\la_r(g^l)\\ [.2cm]
& = \la_l(g^r)(\la_r(g^c)*\la_l(g^c))g^lgg^r(\la_r(g^c)*\la_l(g^c))\la_r(g^l)  \\ [.2cm]
& \xr{*} \la_l(g^r)g^cg^lgg^rg^c\la_r(g^l)\\ [.2cm]
& \xr{*}\la_l(g^r)gg^lgg^rg\la_r(g^l)\xr{*} \la_l(g^r)g\la_r(g^l)
\end{array}$$
and 
$$\be_1(g^rg^cg^l)\xr{*} \la_l(g^r)*(g^rg^cg^l)*\la_r(g^l) = \la_l(g^r)g^c\la_r(g^l)\xr{g^c} \la_l(g^r)g\la_r(g^l)\,.$$ 
Consequently, $\be(g^rg^cgg^cg^l)=\la_l(g^r)g\la_r(g^l)=\be(g^rg^cg^l)$.

$(iii)$. Let $\rho'=\{(u,v)\in \Gxp\times \Gxp\,|\; \be(u)=\be(v)\}$. Note that $\rho'$ is a congruence on $\Gxp$ since $\be(u_1u_2)=\be(\be(u_1)u_2)=\be(u_1\be(u_2))$. Hence, we just need to prove that $\rho=\rho'$. But $\rho\subseteq\rho'$ since $\rho_e\subseteq\rho'$ by $(i)$ and $\rho_s\subseteq\rho'$ by $(ii)$. Assume now that $(u,v)\in\rho'$. Since $u\ap \be(u)$ and $\be(v)\ap v$ by Corollary \ref{m}, we conclude that $(u,v)\in\rho$ and $\rho'\subseteq\rho$.
\end{proof}

Although $\Gx$ and $\rho_e\cup\rho_s$ are always infinite sets (except for $|X|=1$), Proposition \ref{mrhoinv}.$(iii)$ gives us a solution for the word problem for $\ftxm$: to check if $u\ap v$, we just need to compute both $\be(u)$ and $\be(v)$, and check if we get the same mountain. Thus:

\begin{cor}
The word problem for $\ftxm$ is decidable.
\end{cor}

We introduce the operation $\odot$ on both $\mxm$ and $\mx$ as follows:
$$u_1\odot u_2= \be(u_1*u_2)=\be_2(u_1*u_2)\,.$$

\begin{prop}\label{model}
$(\mxm,\odot)$ and $(\mx,\odot)$ are models for $\ftxm$ and $\ftx$, respectively.
\end{prop}

\begin{proof}
This result follows from Proposition \ref{mrhoinv}.$(iii)$ and Lemma \ref{valelem}.$(v)$.
\end{proof}

We already know that $\ftx$ is a regular semigroup, but we intend to show that $\ftx$ is weakly generated by the idempotents of $X$. To do so, we need first to deepen our knowledge about the structure of $\ftx$.

\section{The structure of $\ftx$}

We define the \emph{ground} $\ep(g)$ of a letter $g\in \Gx$ recursively as follows: $\ep(1)=\{1\}$ and $\ep(g)=\ep(g^l)\cup\{g\} \cup\ep(g^r)$ if $\up(g)\geq 1$. Thus $\up(g_1)<\up(g)$ for any $g_1\in\ep(g)\setminus\{g\}$. 

\begin{lem}\label{ground}
Let $g,h\in \Gx$. Then:
\begin{itemize}
\item[$(i)$] $h\in\ep(g)$ \iff\ $\ep(h)\subseteq\ep(g)$.
\item[$(ii)$] If $h\in\ep(g)\setminus\{g\}$, then there exists $u=g_0\cdots g_n\in \lop$ such that $n=\up(g)-\up(h)$, $g_i\in\ep(g)$ for all $0\leq i\leq n$, $g_0=h$ and $g_n=g$.
\end{itemize}
\end{lem} 

\begin{proof}
$(i)$. We only need to prove that $\ep(h)\subseteq\ep(g)$ if $h\in\ep(g)$, and this is done by induction on $\up(g)$. This statement is obviously true for $\up(g)=0$, that is, for $g=1$. Assume that $\ep(h_1)\subseteq\ep(g_1)$ for all $g_1$ such that $\up(g_1)<\up(g)$ and all $h_1\in\ep(g_1)$. If $h\in\ep(g)$, then either $h=g$, or $h\in\ep(g^l)$, or $h\in\ep(g^r)$. In the latter two cases, we have either $\ep(h)\subseteq\ep(g^l)$ or $\ep(h)\subseteq\ep(g^r)$, respectively, by the induction hypothesis. Hence,
$$\ep(h)\subseteq \ep(g^r)\cup \{g\}\cup\ep(g^l)=\ep(g)\,,$$
as wanted.

$(ii)$. Let $h\in\ep(g)\setminus\{g\}$. Once more, we use induction to prove $(ii)$, but now on $n=\up(g)-\up(h)$. If $n=1$, then $h=g^r$ or $h=g^l$. Hence $u=hg$ is an uphill satisfying the conditions of $(ii)$. Assume now that $n\geq 2$. Then $h\in\ep(g^l)$ or $h\in\ep(g^r)$. Without loss of generality, we assume that $h\in\ep(g^l)$. Note that $\up(g^l)-\up(h)=n-1$. Using the induction hypothesis, there exists $g_0\cdots g_{n-1}\in\lop$ such that all $g_i\in\ep(g^l)$ for $0\leq i\leq n-1$, $g_0=h$ and $g_{n-1}=g^l$. Clearly $u=g_0\cdots g_{n-1}g$ is now an uphill satisfying the conditions stated in $(ii)$.
\end{proof}

We write $h\preceq g$ if $h\in\ep(g)$, or equivalently, if $\ep(h)\subseteq\ep(g)$. Note that $\preceq$ is a partial order on $\Gx$ since all letters form $\ep(g)$ have height less than $g$, except $g$ itself.

We extend the notion of ground to any landscape as follows: if $u=g_0\cdots g_n\in \lx$, then $\ep(u)=\cup_{i=0}^n\ep(g_i)$. Clearly, $\ep(u)$ is the union of the grounds of its ridges, and the highest letters of $\ep(u)$ are the peaks of $u$. In particular, if $u\in \mxm$, then $\ep(u)=\ep(\ka(u))$ and all letters from $\ep(u)$ have height less than $\up(\ka(u))$, except $\ka(u)$ of course. By definition, both $\ep(u)$ and $\ep(v)$ are contained in $\ep(u*v)$ if $u$ and $v$ are landscapes such that $\tau(u)=\si(v)$. If $u$ and $v$ are mountain ranges such that $u\to v$, then either $\ep(u)=\ep(v)$ or $\ep(v)$ has one more letter from $\Gx$ than $\ep(u)$, namely the new ridge formed in $v$ by uplifting a river of $u$. Thus $\ep(u)\subseteq\ep(\be(u))=\ep(\ka(\be(u)))$ for any $u\in \lx$.

Given $u\in \mrx$, let $\la_l(u)$ and $\la_r(u)$ be the maximal initial uphill and the maximal final downhill of $u$, respectively (or $\la_l(u)=1=\la_r(u)$ if $u=1$). Thus $\la_l(\be(g))=\la_l(g)$ and $\la_r(\be(g))=\la_r(g)$. By definition of uplifting of rivers, if $v$ is another mountain range, then $\la_l(u)$ is a prefix of $\la_l(\be(u))$; $\la_l(\be(u))$ is a prefix of $\la_l(\be(u*v))$; $\la_r(v)$ is a suffix of $\la_r(\be(v))$; and $\la_r(\be(v))$ is a suffix of $\la_r(\be(u*v))$. Note also that $u_1=\la_l(u_1)*\la_r(u_1)$ for any $u_1\in\mxm$.

We include some particular cases of the conclusions taken above in the following lemma for future reference.

\begin{lem}\label{odot}
Let $u,v,w\in \mxm$ such that $w=u\odot v$. 
\begin{itemize}
\item[$(i)$] $\la_l(u)$ is a prefix of $\la_l(w)$, while $\la_r(v)$ is a suffix of $\la_r(w)$.
\item[$(ii)$] $\up(w)\geq\max\{\up(u),\up(v)\}$, and $\up(w)=\up(u)$ {\rm [}$\up(w)=\up(v)${\rm ]} \iff\ $\ka(w)=\ka(u)$ {\rm [}$\ka(w)=\ka(v)${\rm ]}.
\item[$(iii)$] $\ep(u)\cup\ep(v)\subseteq\ep(w)$.
\end{itemize}
\end{lem}

\begin{proof}
$(i)$ and $(iii)$ are just particular cases of the observations made above, and $(ii)$ follows obviously from $(i)$.
\end{proof}

\begin{prop}\label{desR}
Let $u,v\in \mxm$. Then:
\begin{itemize}
\item[$(i)$] $u\leq_{\Rc} v$ \iff\ $\la_l(v)$ is a prefix of $\la_l(u)$. Thus $v$ covers $u$ for $\leq_{\Rc}$ \iff\ $\la_l(u)=\la_l(v)\ka(u)$.
\item[$(ii)$] $u\leq_{\Lc} v$ \iff\ $\la_r(v)$ is a suffix of $\la_r(u)$. Thus $v$ covers $u$ for $\leq_{\Lc}$ \iff\ $\la_r(u)=\ka(u)\la_r(v)$.
\item[$(iii)$] $u\leq_{\Jc} v$ \iff\ $\ka(v)\preceq\ka(u)$. Thus $v$ covers $u$ for $\leq_{\Jc}$ \iff\ $\ka(v)\in\{(\ka(u))^l,(\ka(u))^r\}$.
\end{itemize}
\end{prop}

\begin{proof}
Note that, in the three items of this proposition, the second part of each one is an immediate consequence of the first part. Furthermore, $(ii)$ is the dual of $(i)$. Hence, we shall prove only the first part of $(i)$ and $(iii)$.

$(i)$. Assume that $u\leq_{\Rc} v$. Thus $u=v\odot w$ for some $w\in\mxm$. By Proposition \ref{odot}.$(i)$, $\la_l(v)$ is a prefix of $\la_l(u)$. Assume now that $\la_l(v)$ is a prefix $\la_l(u)$. Then $u=\la_l(v)*u_1$ for some landscape $u_1$ with $\tau(\la_l(v))=\ka(v)=\si(u_1)$. Let $v_1=\la_r(v)$ and consider $w=\cevm{v_1}*u_1$. Clearly $w\in\mxm$ and $v_1*(\cevm{v_1})\xr{*} \ka(v)$. Hence
$$v*w=\la_l(v)*v_1*\cevm{v_1}*u_1\xr{*}\la_l(v)*\ka(v)*u_1=\la_l(v)*u_1=u\,,$$
that is, $u=v\odot w$ and $u\leq_{\Rc} v$. 

$(iii)$. If $u\leq_{\Jc} v$, then $u=w_1\odot v\odot w_2$ for two mountains $w_1$ and $w_2$. By Lemma \ref{odot}.$(iii)$, $\ep(\ka(v))=\ep (v)\subseteq\ep(u)=\ep(\ka(u))$ and $\ka(v)\preceq\ka(u)$. Conversely, if $\ka(v)\preceq\ka(u)$, then let $u_1$ be an uphill such that $\si(u_1)=\ka(v)$ and $\tau(u_1)=\ka(u)$, whose existence is guaranteed by Lemma \ref{ground}.$(ii)$. Observe that $w_1=\la_l(u)*\cevm{u_1}*\cevl{\la_l(v)}$ and $w_2=\cevl{\la_r(v)}*u_1*\la_r(u)$ are well defined mountains. Further,
$$\begin{array}{ll}
w_1*v*w_2\!\!\! & =\la_l(u)*\cevm{u_1}*(\cevl{\la_l(v)}*\la_l(v))*(\la_r(v)*\cevl{\la_r(v)})*u_1*\la_r(u) \\ [.2cm]
& \xr{*} \la_l(u)*\cevm{u_1}*\ka(v)*\ka(v)*u_1*\la_r(u) \\ [.2cm]
& \xr{*} \la_l(u)*\ka(u)*\la_r(u)=u\,,
\end{array}$$
and $w_1\odot v\odot w_2=u$. Consequently, $u\leq_{\Jc} v$.
\end{proof}

\begin{cor}\label{R}
Let $u,v\in\mxm$. Then
\begin{itemize}
\item[$(i)$] $u\Rc v$ \iff\ $\la_l(u)=\la_l(v)$.
\item[$(ii)$] $u\Lc v$ \iff\ $\la_r(u)=\la_r(v)$.
\item[$(iii)$] $u\Jc v$ \iff\ $\ka(u)=\ka(v)$.
\item[$(iv)$] $u\Hc v$ \iff\ $u=v$.
\item[$(v)$] $\Dc =\Jc$.
\end{itemize}
\end{cor}

\begin{proof}
The first three statements follow from the corresponding statements of Proposition \ref{desR}. $(iv)$ is a consequence of $(i)$ and $(ii)$. So, we only need to prove that $\Jc\subseteq\Dc$. Assume that $u\Jc v$. Then $w=\la_l(u)*\la_r(v)\in \mxm$ since $\ka(u)=\ka(v)$ by $(iii)$. Now, by $(i)$ and $(ii)$, we conclude that $u\Rc w\Lc v$, whence $u\Dc v$.
\end{proof}

The previous corollary tells us that the $\Dc=\Jc$-classes of $\ftxm$ are in one-to-one correspondence with the elements of $\Gx$, and that the set $\{[g]\,|\;g\in\Gx\}$ is a transversal (or cross-section) for the set of $\Dc$-classes of $\ftxm$. The next result gives us the size of each $\Rc$, $\Lc$ and $\Dc$-class of $\ftxm$.

\begin{cor}\label{size}
If $g\in \Gix$ for $i\geq 1$, then $|\Rcc_{[g]}|=2^{i-1}=|\Lcc_{[g]}|$ and $|\Dcc_{[g]}|=2^{2i-2}$. Further, $\Dcc_{[g]}$ has $2^{i-1}$ $\Rc$-classes and $2^{i-1}$ $\Lc$-classes. 
\end{cor}

\begin{proof}
We only need to prove that $|\Rcc_{[g]}|=2^{i-1}$ since $|\Lcc_{[g]}|=2^{i-1}$ follows by duality and the statements about $\Dc$ follow from the statements about $\Rc$ and $\Lc$ and from Corollary \ref{R}.$(iv)$.

From Corollary \ref{R}.$(i)$, the size of $\Rcc_{[g]}$ is equal to the number of downhills from $g$ to $1$. Since $g\in \Gix$, each downhill from $g$ to $1$ is of the form $g_0g_1\cdots g_i$ where $g_0=g$, $g_i=1$ and $g_j\in\{g_{j-1}^l,g_{j-1}^r\}$ for $1\leq j\leq i-1$. Moreover, $g_{j-1}^l\neq g_{j-1}^r$ for $1\leq j\leq i-1$ by definition of $\Gx$. Hence, there are $2^{i-1}$ downhills from $g$ to $1$, and $|\Rcc_{[g]}|=2^{i-1}$.
\end{proof}

We have already seen that $[g]\in S([g^rg^c],[g^cg^l])$ in Lemma \ref{prodelem}.$(iv)$. The next result tells us more. It tells us that $S([g^rg^c],[g^cg^l])$ has no other elements.

\begin{lem}\label{sandwich}
Let $g\in \Gx$ such that $\up(g)\geq 2$. Then 
$$S([g^rg^c],[g^cg^l])=\{[g]\}$$
in $\ftxm$.
\end{lem}

\begin{proof}
Let $u\in\mxm$ be such that $[u]\in S([g^rg^c],[g^cg^l])$. Then $\ka(u)=g$ by Corollary \ref{R} and since $[u]\Dc [g]$. Let $v$ be the mountain $\la_l(g^c)\la_r(g^l)$. By Lemma \ref{valelem}.$(iii)$, $\be_1(g^cg^l)\xr{*} v$, and so $g^cg^l\ap v$. Hence,
$$u\ap g^cg^lu\ap vu\ap v\odot u\,,$$
and by Lemma \ref{odot}.$(i)$ and Proposition \ref{mrhoinv}.$(iii)$, $\la_l(v)=\la_l(g^c)g^l$ is a prefix of $\la_l(u)$. Therefore, $\la_l(u)=\la_l(g^c)g^lg$ because $\ka(u)=g$. We have shown that $\la_l(u)=\la_l(g)$. Dually, we conclude also that $\la_r(u)=\la_r(g)$. Consequently, $u=\be_1(g)\ap g$ and $S([g^rg^c],[g^cg^l])=\{[g]\}$.
\end{proof}

We have now all the ingredients necessary to show that $\ftx$ is weakly idempotent generated by $X$.

\begin{prop}\label{weakly}
$\ftx$ is weakly idempotent generated by $X$.
\end{prop}

\begin{proof}
Let $S$ be a regular subsemigroup of $\ftx$ containing $[X]$. Since $[\Gx]$ generates $\ftxm$ and $[1]$ is the identity element of $\ftxm$, it is enough to prove that $[\Gx]\subseteq S\cup\{[1]\}$ to conclude that $S=\ftx$ and that $\ftx$ is weakly idempotent generated by $X$. We show that $[\Gix]\subseteq S$, for $i\geq 1$, by induction on $i$.

Clearly $[\G_1(X)]=[X]\subseteq S$ by definition of $S$. Let $i\geq 2$ and assume that $[\G_j(X)]\subseteq S$ for all $1\leq j< i$. So, if $g\in \Gix$, then $[g^l],[g^c],[g^r]\in S\cup\{[1]\}$. Consequently, also $[g^cg^l]$ and $[g^rg^c]$ belong to $S$. Now, since $S$ is a regular semigroup, $S([g^rg^c],[g^cg^l])\cap S\neq\emptyset$. By the previous lemma, $[g]\in S$ and $[\Gix]\subseteq S$.
\end{proof}

\begin{cor}\label{weaklywg}
Let $\phi':X\to \ftx$ be the ``identity'' mapping. Then 
$$(\ftx,\phi':X\to\ftx)\in\wg\,.$$
\end{cor}

We have already remarked that regular semigroups weakly generated by  a set of idempotents are idempotent generated. However, we cannot guarantee that they are generated by a finite set of idempotents even if they are weakly generated by a finite set of idempotents. For example, $\ftn$ is weakly generated by $n$ idempotents, but it is not generated by any finite set of idempotents if $n\geq 2$.

Let us turn now to the identification of idempotents and of inverses of elements of $\ftxm$. For that purpose we need the notion of a gorge. 
\begin{description}
\item[Gorge] canyon $w$ such that $w\xr{*}\si(w)=\tau(w)$.
\end{description}
Note that we have already encounter gorges earlier. For example, in Lemma \ref{updownhill}, we prove that $u\cev{u}$ is a gorge for every downhill $u$. Another example appears in Lemma \ref{valelem}.$(i)$, where we prove that $\la_r(g)*\la_l(g)$ is a gorge too. Note also that if $w$ is a gorge and $w\xr{*}w_1$, then $w_1\xr{*}\si(w)$, but we cannot guarantee that $w_1$ is a gorge since it may not be a valley. In the next result we characterize the idempotents, the inverses of an element, and the natural partial order on $\mxm$ using this notion of a gorge.

\begin{prop}\label{idgorges}
If $u,v\in\mxm$, then:
\begin{itemize}
\item[$(i)$] $u$ is an idempotent \iff\ the canyon $\la_r(u)*\la_l(u)$ is a gorge.
\item[$(ii)$] $v$ is an inverse of $u$ \iff\ $\la_r(u)*\la_l(v)$ and $\la_r(v)*\la_l(u)$ are both gorges.
\item[$(iii)$] $v<u$ \iff\ $\la_l(v)=\la_l(u)\,u_1$ and $\la_r(v)=u_2\,\la_r(u)$ for some $u_1,u_2\in \Gxp$ such that $u_2\,\ka(u)\,u_1$ is a gorge.
\end{itemize}
\end{prop}

\begin{proof}
$(i)$. Since $u=\la_l(u)*\la_r(u)$ and $u^2= \la_l(u)*\la_r(u)*\la_l(u)*\la_r(u)$, it is obvious that $u^2\approx u$ \iff\ $\la_r(u)*\la_l(u)\xr{*} \ka(u)$, that is, \iff\ the canyon $\la_r(u)*\la_l(u)$ is a gorge.

$(ii)$. Note that if $v$ is an inverse of $u$, then $\ka(v)=\ka(u)$, and so $\la_r(u)*\la_l(v)$ and $\la_r(v)*\la_l(u)$ are canyons. The proof of $(ii)$ is now obvious as in $(i)$.

$(iii)$. Assume first that $v<u$. In particular, we have also $v\leq_{\Rc} u$ and $v\leq_{\Lc} u$. Thus $\la_l(v)=\la_l(u)u_1$ and $\la_r(v)=u_2\la_r(u)$, for some $u_1,u_2\in \Gxp$, by Proposition \ref{desR} and since $u\neq v$. But, there exists also another mountain $w\in\mxm$ such that
$$w\in E(\mxm),\quad w\Lc v\quad\mbox{ and }\quad v=u\odot w$$
as $v\leq u$. Hence, $\la_r(w)*\la_l(w)$ is a gorge by $(i)$, $\la_r(w)=\la_r(v)$ by Corollary \ref{R}.$(ii)$, and $\la_r(u)*\la_l(w)\xr{*} \ka(u)\,u_1$. Therefore
$$\la_r(w)*\la_l(w)=u_2\,\la_r(u)*\la_l(w)\xr{*} u_2\,\ka(u)\,u_1\,,$$
and consequently $u_2\,\ka(u)\,u_1$ is a gorge because $\la_r(w)*\la_l(w)$ is a gorge too (note that $u_2\,\ka(u)\,u_1$ is a canyon).

Assume now that $\la_l(v)=\la_l(u)\,u_1$ and $\la_r(v)=u_2\,\la_r(u)$ for some $u_1,u_2\in \Gxp$ such that $u_2\,\ka(u)\,u_1$ is a gorge. In particular $u\neq v$. Note also that
$$w_r=\cevl{\la_r(u)}\,u_1*\la_r(v)\quad\mbox{ and }\quad w_l=\la_l(v)*u_2\,\cevl{\la_l(u)}$$
are well defined mountains. In fact, they are both idempotents of $\mxm$. For example,
$$\la_r(v)*\cevl{\la_r(u)}\,u_1 =u_2\,\la_r(u)*\cevl{\la_r(u)}\,u_1\xr{*}u_2\, \ka(u)\, u_1\,;$$
thus $\la_r(v)*\cevl{\la_r(u)}\,u_1$ is a gorge and so $w_r$ is an idempotent. Finally, observe that $u\odot w_r=v$ and $w_l\odot u=v$, whence $v<u$.
\end{proof}

Identifying which canyons are gorges is not an easy task. We leave a deeper study of the gorges for later. Some more information about the structure of $\ftx$ will be gathered from that study. Meanwhile, let us analyze, in the next section, some universal properties that $\ftx$ has in relation with the category $\wg$.

\section{Some universal properties of $\ftx$}\label{sec5}

Let $S$ be a semigroup. A \emph{skeleton mapping} is a mapping $\phi:\Gx\to E(S^1)$ such that
\begin{itemize}
\item[$(i)$] $\phi_{|X}$ is a one-to-one mapping such that $X\phi\subseteq E(S)$;
\item[$(ii)$] $(1\phi)(g\phi)=g\phi=(g\phi)(1\phi)$ for all $g\in\G_1(X)$;
\item[$(iii)$] $g\phi\in S((g^r\phi)(g^c\phi),(g^c\phi)(g^l\phi))$
for all $g\in \Gix$ with $i\geq 2$.
\end{itemize}
Note that, by $(iii)$, $g\phi=(g^c\phi)(g^l\phi)(g\phi)=(g\phi)(g^r\phi)(g^c\phi)$ for any $g\in\G_i(X)$ with $i\geq 2$. It is now trivial to conclude that $(1\phi)(g\phi)=g\phi=(g\phi)(1\phi)$, for any $g\in\Gix$ with $i\geq 0$, by induction on $i$. 

\begin{prop}\label{r52}
If $\phi:\Gx\to E(S^1)$ is a skeleton mapping, then there is a unique (semigroup) homomorphism $\varphi:\ftxm\to S^1$ extending $\phi$ (that is, such that $g\phi=[g]\varphi$ for all $g\in\Gx$). Furthermore, $\langle(\Gx)\phi\rangle=(\ftxm)\varphi$ is a regular monoid with identity element $1\phi$.
\end{prop}

\begin{proof}
Let $\phi_1$ be the unique homomorphism $\phi_1:\Gxp\to S^1$ that extends $\phi$. Note that $\phi_1$ exists, and it is unique, since $\Gxp$ is freely generated by $\Gx$. Moreover, $\rho_e$ is contained in the kernel of $\phi_1$ by the observation made above and since $(\Gx)\phi\subseteq E(S^1)$, while $\rho_s$ is contained in the kernel of $\phi_1$ by $(iii)$. Thus, if $\varphi_1$ denotes the natural quotient homomorphism $\varphi_1:\Gxp\to\ftxm$, then there exists a homomorphism $\varphi:\ftxm\to S^1$ such that $\phi_1=\varphi_1\varphi$. Hence, $\varphi$ is a homomorphism extending $\phi$. If $\varphi':\ftxm\to S^1$ is another homomorphism extending $\phi$, then $\varphi_1\varphi'=\phi_1=\varphi_1\varphi$, and so $\varphi'=\varphi$ since $\varphi_1$ is surjective. Finally,  the second part of this result follows immediately from the fact that $\ftxm$ is a regular monoid with identity element [1] (this property is inherited by homomorphic images).
\end{proof}

We remark that, in the previous proposition, $\ftxm\varphi$ is a regular monoid and a subsemigroup of $S^1$. However, it may not be a submonoid of $S^1$. Notice that $[1]\varphi$ may not be the identity of element of $S^1$, but it works as an identity element for the subsemigroup $(\ftxm)\varphi$.

Note that $X\phi\subseteq E(S)$ by definition of skeleton mapping. Now, by $(iii)$, we can conclude that $g\phi\in E(S)$ for any $g\in G_i(X)$ with $i\geq 1$. Thus $(\Gx\setminus\{1\})\phi\subseteq E(S)$. A \emph{skeleton} of $S$ is the image of $\Gx\setminus\{1\}$ under some skeleton mapping $\phi:\Gx\to E(S^1)$. Hence, a skeleton of $S$ is a subset of $E(S)$ with ``some structure''.

\begin{cor}\label{r53}
If $A$ is a skeleton of $S$, then $\langle A\rangle$ is a regular subsemigroup of $S$.
\end{cor}

\begin{proof}
Let $X$ be a nonempty set and let $\phi:\Gx\to S^1$ be a skeleton mapping such that $A=(\Gx\setminus\{1\})\phi$. By the previous proposition, $T=\langle (\Gx)\phi\rangle$ is a regular subsemigroup of $S^1$ with an identity element $1\phi$. Thus $\langle A\rangle=T$ or $\langle A\rangle=T\setminus\{1\phi\}$. Consequently, $\langle A\rangle$ is a regular subsemigroup of $S$.
\end{proof}

\begin{lem}\label{r51}
If $S$ is a regular semigroup, then each one-to-one mapping $\phi':X\to E(S)$ can be recursively extended into a skeleton mapping $\phi:\Gx\to E(S^1)$.
\end{lem}

\begin{proof}
We define $\phi$ recursively as follows. We begin by setting $1\phi=1$ and $x\phi=x\phi'$ for all $x\in X$. Let $i\geq 2$ and assume that $\phi$ is already well defined for all $g_1\in \G_j(X)$ with $j<i$. Let $g\in \Gix$. By our assumption, $g^l\phi$, $g^c\phi$ and $g^r\phi$ are already well defined. Since $S$ is a regular semigroup, the sandwich set $S((g^r\phi)(g^c\phi),(g^c\phi)(g^l\phi))$ is nonempty. Hence, choose $h\in S((g^r\phi)(g^c\phi),(g^c\phi)(g^l\phi))$ and set $g\phi=h$. By construction, the mapping $\phi:\Gx\to S^1$ becomes a skeleton mapping.
\end{proof}

The construction procedure described in the previous proof has a choice factor and, therefore, we usually have many distinct skeleton mappings extending $\phi'$. In terms of skeletons, if $S$ is regular and $\phi':X\to E(S)$ is a one-to-one mapping, then $S$ has at least one skeleton $A=(\Gx\setminus\{1\})\phi$ obtained from a skeleton mapping $\phi:\Gx\to E(S^1)$ extending $\phi'$, but usually it has more. We shall say that $A$ is a skeleton induced by $\phi'$ (or by $\phi$).

From Lemma \ref{r51} and Proposition \ref{r52}, it is also obvious that for any $(S,\phi':X\to S)\in\wg$, there exists a morphism $\varphi:\ftx\to S$ extending $\phi'$ (note that $\phi'$ is a one-to-one mapping such that $X\phi'\subseteq E(S)$). The next corollary gives us a slightly stronger result.

\begin{cor}
Let $S$ be a semigroup and let $T$ be a regular subsemigroup of $S$ weakly generated by $X\phi'$ for some one-to-one mapping $\phi':X\to E(S)$. Then $T$ is the image of $\ftx$ under some homomorphism $\varphi:\ftx\to S$ that extends $\phi'$.
\end{cor}

\begin{proof}
Note that $\phi'$ is also a one-to-one mapping from $X$ into $E(T)$. Applying Lemma \ref{r51} and Proposition \ref{r52} to the mapping $\phi':X\to E(T)$, we obtain a homomorphism $\varphi:\ftxm\to T^1$ extending $\phi'$. Then $(\ftx)\varphi$ is a regular subsemigroup of $T$ containing $X\phi'$, and consequently $(\ftx)\varphi=T$ since $T$ is weakly generated by $X\phi'$.  Finally, observe that $\varphi_{|\ftx}$ can be viewed as a homomorphism from $\ftx$ to $S$.
\end{proof}

The next proposition is now an obvious consequence of the previous corollary.

\begin{prop}\label{fiwig}
All objects from $\wg$ are morphic images of $\ftx\in\wg$. 
\end{prop}

It is natural to ask now if the converse of the last proposition holds also. In other words,
\begin{quote}
if $\varphi:\ftx\to S$ is a homomorphism whose restriction to $[X]$ is a one-to-one mapping such that $[X]\varphi\subseteq E(S)$, then $((\ftx)\varphi,\varphi_{|X})\in\wg$?
\end{quote}
Unfortunately, the answer to this question is negative for any nontrivial set $X$. In the next example we construct a semigroup $T_1$ that shows the answer is negative for $|X|=2$. The construction process can be easily adapted for any $X$ with more than two elements. 
\vspace*{.3cm}

\noindent {\bf Example 1}: Let $X=\{e,f\}$ and recall the sets $\G_1(X)$, $\G_2(X)$ and $\G_3(X)$ from page \pageref{g123}. Consider now the two sets
$$R=[X]\cup\Dcc_{[e_1]}\cup\Dcc_{[f_1]}\cup\Dcc_{[h]}\quad\mbox{ and }\quad I=\ftt\setminus R\,.$$
Since $[u]=[\be(u)]\Dc [\ka(\be(u))]$, then  $[u]\in I$ \iff\ $\ka(\be(u))=g$ for some $g\in\Gx$ such that $\up(g)\geq 3$ and $g\neq h$. Thus $\ep(\be(u))=\ep(g)$, and so $[u]\in I$ \iff\ $\ep(\be(u))\cap\{h_1,h_2,h_3\}\neq\emptyset$ (note that $\up(g)\geq 3$ with $g\neq h$ \iff\ $\ep(g)\cap\{h_1,h_2,h_3\}\neq\emptyset$). We can now conclude that $I$ is an ideal of $\ftt$ since $\ep(v)\cup\ep(w)\subseteq \ep(vw)$.

Consider the quotient semigroup $\ftt/I$. Then $\ftt/I$ is obtained from $\ftt$ by setting $[h_1]\approx [h_2]\approx [h_3]\approx 0$, where $0$ represents a zero element for the new semigroup. In fact, we can consider instead $[f_1e_1]\approx [e_1f_1]\approx [fef]\approx 0$ since one can check that $[f_1e_1]\in\Dcc_{[h_1]}$, $[e_1f_1]\in\Dcc_{[h_2]}$ and $[fef]\in\Dcc_{[h_3]}$. Therefore, $\ftt/I$ is isomorphic to the semigroup $T$ generated by five idempotents $\{e,f,e_1,f_1,h\}$ subject to the following relations:
$$e_1\in S(f,e),\; f_1\in S(e,f),\; h\in S(f_1f,fe_1)\;\mbox{ and }\; f_1e_1= e_1f_1= fef= 0.$$
We present the ``Egg-box'' diagram of the $\Dc$-classes of $T$ in Figure \ref{figT} (we can compute it either manually or with the GAP software \cite{gap,mitchell}; as usual, the symbol $^*$ identifies the idempotents).

\begin{figure}[ht]
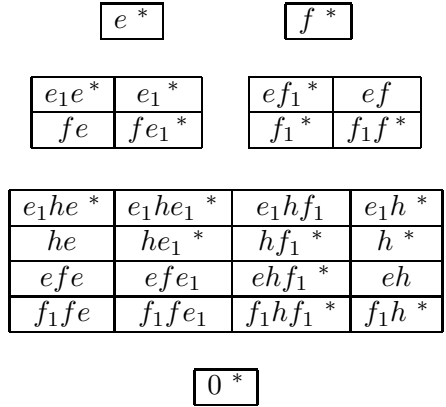

\begin{tabular}{c}
	\begin{tabular}{|c|}
	\hline	$e\ ^*$\\ \hline 
	\end{tabular}\hspace*{1.5cm}
	\begin{tabular}{|c|}
	\hline	$f\ ^*$\\ \hline
\end{tabular} \\ [.5cm]	
\begin{tabular}{|c|c|}
	\hline $e_1e\, ^*$ & $e_1\, ^*$ \\
	\hline $fe$ & $fe_1\, ^*$ \\ \hline
\end{tabular}\hspace*{.5cm}
\begin{tabular}{|c|c|}
	\hline $ef_1\, ^*$ & $ef$ \\
	\hline $f_1\, ^*$ & $f_1f\, ^*$ \\ \hline
\end{tabular} \\ [.8cm]	
\begin{tabular}{|c|c|c|c|}
	\hline	$e_1he\ ^*$ & $e_1he_1\ ^*$ & $e_1hf_1$ & $e_1h\ ^*$ \\ 
	\hline	$he$ & $he_1\ ^*$ & $hf_1\ ^*$ & $h\ ^*$ \\ 
	\hline 	$efe$ & $efe_1$ & $ehf_1\ ^*$ & $eh$ \\
	\hline $f_1fe$ & $f_1fe_1$ & $f_1hf_1\ ^*$ & $f_1h\ ^*$ \\ \hline
	\end{tabular}\\ [1.2cm]
\begin{tabular}{|c|}
	\hline	$0\ ^*$\\ \hline
\end{tabular}
\end{tabular}
	\caption{``Egg-box'' diagram of the $\Dc$-classes of $T$.}\label{figT}
\end{figure}

Let $T_1=T/\theta$ where $\theta$ is the congruence on $T$ generated by 
$$\{(fe,he),(ef,eh)\}\,.$$ 
The ``Egg-box'' diagram of the $\Dc$-classes of $T_1$ is depicted in Figure \ref{figT1} (left diagram). Basically, $T_1$ is obtained from $T$ by ``merging'' $\Dcc_{e_1}$ and $\Dcc_{f_1}$ into $\Dcc_h$. Then $T_1$ is a homomorphic image of $\ftt$ under a homomorphism whose restriction to $X$ is the identity mapping. However, $T_1$ is not weakly generated by $X$. The subsemigroup $T_2$ of $T_1$ generated by $\{e,f,e_1hf_1\}$ is a proper regular subsemigroup containing $X$. In fact, $T_2$ is obtained from $T_1$ by deleting $\Lcc_{e_1}$ and $\Rcc_{f_1}$. See the ``Egg-box'' diagram representation for the $\Dc$-classes of $T_2$ in Figure \ref{figT1} (right diagram).
\begin{figure}[ht]
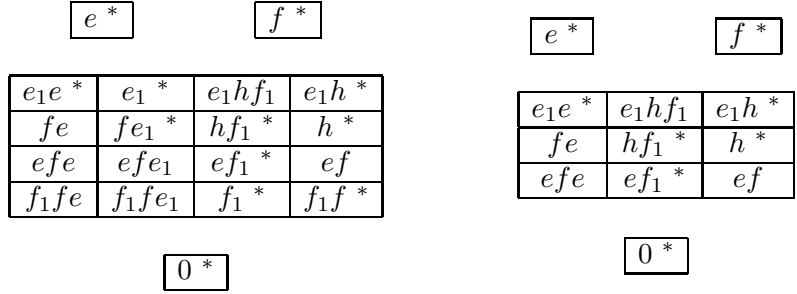

\begin{subfigure}[b]{.42\linewidth}
\begin{tabular}{c}
	\begin{tabular}{|c|}
	\hline	$e\ ^*$\\ \hline 
	\end{tabular}\hspace*{1.5cm}
	\begin{tabular}{|c|}
	\hline	$f\ ^*$\\ \hline
 	\end{tabular} \\ [.5cm]			
	\begin{tabular}{|c|c|c|c|}
	\hline	$e_1e\ ^*$ & $e_1\ ^*$ & $e_1hf_1$ & $e_1h\ ^*$ \\ 
	\hline	$fe$ & $fe_1\ ^*$ & $hf_1\ ^*$ & $h\ ^*$ \\ 
	\hline 	$efe$ & $efe_1$ & $ef_1\ ^*$ & $ef$ \\
	\hline $f_1fe$ & $f_1fe_1$ & $f_1\ ^*$ & $f_1f\ ^*$ \\ \hline
	\end{tabular}\\ [1.2cm]
	\begin{tabular}{|c|}
	\hline	$0\ ^*$\\ \hline
	\end{tabular}
\end{tabular}
\end{subfigure}	\hspace*{1.2cm}
\begin{subfigure}[b]{.32\linewidth}
\begin{tabular}{c}
	\begin{tabular}{|c|}
	\hline	$e\ ^*$\\ \hline 
	\end{tabular}\hspace*{1.5cm}
	\begin{tabular}{|c|}
	\hline	$f\ ^*$\\ \hline
\end{tabular} \\ [.5cm]	
\begin{tabular}{|c|c|c|}
	\hline	$e_1e\ ^*$ & $e_1hf_1$ & $e_1h\ ^*$ \\ 
	\hline	$fe$ & $hf_1\ ^*$ & $h\ ^*$ \\ 
	\hline 	$efe$ & $ef_1\ ^*$ & $ef$ \\ \hline
	\end{tabular}\\ [1cm]
\begin{tabular}{|c|}
	\hline	$0\ ^*$\\ \hline
\end{tabular}
\end{tabular}
\end{subfigure}
	\caption{``Egg-box'' diagrams of the $\Dc$-classes of $T_1$ (left) and $T_2$ (right).}\label{figT1}
\end{figure}
\hfill\qed\vspace*{.3cm}

The semigroup $T_1$ of the previous example also show us that being generated by a skeleton, induced by a one-to-one mapping $\phi':X\to E(S)$, is not sufficient for a semigroup $S$ to be weakly generated by $X$. The next result tells us that if $\phi'$ induces a unique skeleton $A$ and $A$ generates $S$, then $S$ is weakly generated by $X$. In particular, if $\phi'$ induces a unique skeleton mapping and its image generates $S$, then $S$ is weakly generated by $X$. Note that having a unique skeleton is a weaker condition than having a unique skeleton mapping: in the Example 2 below there are four possible skeleton mappings and we can easily find two of them giving rise to the same skeleton. However, the Example 2 is constructed with a different purpose: to show that there are regular semigroups $S$ weakly generated by $X\phi'$, for some one-to-one mapping $\phi':X\to E(S)$, with more than one skeleton induced by $\phi'$. In other words, having a unique skeleton induced by $\phi':X\to E(S)$, and obviously being generated by that skeleton, is not a necessary condition for being weakly generated by the set of idempotents $X\phi'$. 

\begin{prop}
Let $S$ be a semigroup and $\phi':X\to E(S)$ be a one-to-one mapping. If $\phi'$ induces a unique skeleton $A$, then $S$ has the smallest regular subsemigroup containing $X\phi'$, namely the subsemigroup $\langle A\rangle$ generated by $A$. In particular, if $\langle A\rangle=S$ also, then $(S,\phi')\in\wg$.
\end{prop}

\begin{proof}
Note that $\langle A\rangle$ is a regular subsemigroup of $S$ containing $X\phi'$ by Corollary \ref{r53}. Let $T$ be another regular subsemigroup of $S$ containing $X\phi'$. Since $T$ is regular, there exists a skeleton mapping $\phi:\Gx\to T^1$ extending $\phi':X\to E(T)$ by Lemma \ref{r51}. 
Thus $(\Gx\setminus\{1\})\phi=A$ since $(\Gx\setminus\{1\})\phi$ is also a skeleton of $S$ induced by $\phi'$. Consequently $\langle A\rangle \subseteq T$. In particular, $\langle A\rangle$ is weakly generated by the idempotents $X\phi'$, and the second part of this proposition is now obvious.
\end{proof}

\noindent {\bf Example 2}: Let $R$ be the semigroup generated by the set $\{e,f,g,g_1\}$ of four idempotents subject to the relations
$$g\in S(f,e),\;\; g_1\in S(e,f),\;\; efe=geg_1,\;\; fef=g_1ge\;\;\mbox{ and }\;\;efg=gg_1f\,.$$
The ``Egg-box'' diagram representation of the $\Dc$-classes of $R$ is depicted in Figure \ref{figR}. Clearly, $R$ is weakly idempotent generated by $X=\{e,f\}$.
\begin{figure}[ht]
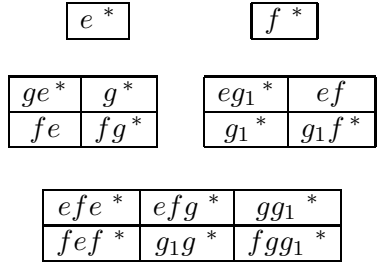

\begin{tabular}{c}
	\begin{tabular}{|c|}
	\hline	$e\ ^*$\\ \hline 
	\end{tabular}\hspace*{1.5cm}
	\begin{tabular}{|c|}
	\hline	$f\ ^*$\\ \hline
\end{tabular} \\ [.5cm]	
\begin{tabular}{|c|c|}
	\hline $ge\, ^*$ & $g\, ^*$ \\
	\hline $fe$ & $fg\, ^*$ \\ \hline
\end{tabular}\hspace*{.5cm}
\begin{tabular}{|c|c|}
	\hline $eg_1\, ^*$ & $ef$ \\
	\hline $g_1\, ^*$ & $g_1f\, ^*$ \\ \hline
\end{tabular} \\ [.8cm]	
\begin{tabular}{|c|c|c|}
	\hline	$efe\ ^*$ & $efg\ ^*$ & $gg_1\ ^*$\\ 
	\hline $fef\ ^*$ & $g_1g\ ^*$ & $fgg_1\ ^*$\\ \hline
\end{tabular}
\end{tabular}
	\caption{``Egg-box'' diagram of the $\Dc$-classes of $R$.}\label{figR}
\end{figure}

Let $\phi:\Gx\to R^1$ be a skeleton mapping extending the identity mapping $\phi':X\to E(R)$, and recall the notations $\G_2(X)=\{e_1,f_1\}$ and $\G_3(X)=\{h,h_1,h_2,h_3\}$. Note that $e_1\phi=g$ and $f_1\phi=g_1$ since $S(f,e)=\{g\}$ and $S(e,f)=\{g_1\}$, respectively, in $R$. Further, $h_2\phi=g_1g$ and $h_3\phi=efe$ because $S(gf,fg_1)=\{g_1g\}$ and $S(ge,eg_1)=\{efe\}$, respectively. However, we have two possible choices for both $h\phi$ and $h_1\phi$:
$$h\phi\in S(g_1f,fg)=\{g_1g,fef\}\;\;\mbox{ and }\;\;h_1\phi\in S(g_1e,eg)=\{efe,gg_1\}\,.$$

If we choose $h\phi=g_1g$ and $h_1\phi=efe$, then $(\G_2(X))\phi=\{efe,g_1g\}$ and it is trivial to check that $(\G_3(X))\phi=\{fef,efg,fgg_1\}$. Then, by induction on $i\geq 2$, we can easily prove that
$$(\Gix)\phi=\left\{\begin{array}{ll}\;\{efe,g_1g\} & \mbox{ if } i \mbox{ even} \\ [.2cm]
\;\{fef,efg,fgg_1\} & \mbox{ if } i \mbox{ odd,}
\end{array}\right.$$
since $\Dcc_{efe}$ is a rectangular band. Therefore, 
$$A=\{e,f,g,g_1,efe,efg,fef,g_1g,fgf_1\}$$ 
is a skeleton of $R$. But now, if we choose $h\phi=fef$ and $h_1\phi=gg_1$ instead, we get another skeleton: $A\cup\{gg_1\}$. Hence, $R$ is weakly generated by $X=\{e,f\}$ but has more than one skeleton.\hfill\qed\vspace*{.3cm}

From the previous results and examples, some questions arise naturally. When does an image of $\ftx$, under a homomorphism whose restriction to $[X]$ is an idempotent one-to-one mapping, belong to $\wg$? Is this question decidable? Is there any necessary and sufficient condition for a semigroup to be weakly idempotent generated by $X$? Is it decidable if a semigroup is weakly idempotent generated by $X$, even if we restrict ourselves to finite semigroups? If $S$ is a finitely idempotent generated semigroup, can we find (good bounds for) the size of a minimal set $X$ that weakly idempotent generates $S$? These are just some examples of questions that now can be posed and are still open.

\section{weakly finitely idempotent generated semigroups}

Let $m,n\geq 2$. In this section we prove, somehow surprising, that $\mft$ can be embedded into $\ftn$ even if $m>n$. Note that $\mft$ is obviously a subsemigroup of $\ftn$ if $m\leq n$. In fact, we shall prove that the monoid $\mft^1$ is isomorphic to a subsemigroup of $\ftn$. In particular, all monoids $\mft^1$ and all semigroups $\mft$ are isomorphic to subsemigroups of $\ftt$. As a consequence, we show that all regular semigroups weakly generated by a finite set of idempotents (which includes all finitely idempotent generated regular semigroups) strongly divide $\ftt$, that is, they are homomorphic images of regular subsemigroups of $\ftt$. 

For each $g\in \Gix$ with $i\geq 1$, let 
$$\G(g)=\{g_1\in \Gx\,:\;g_1^l=g\mbox{ or } g_1^r=g\}\,.$$
Note that $\G(g)$ is a subset of $\G_{i+1}(X)$. We begin by computing the cardinality of each sets $\G(g)$ and use this computation to calculate, recursively, the cardinality of each $\Gix$.

\begin{lem}\label{sizeG}
Let $X$ be a set of cardinality $n\geq 2$ and let $g\in \Gix$ for some $i\geq 1$. Then  $|\G(g)|=\frac{2^{2i-1}(3n-5)+4}{3}$ and $|\G_{i+1}(X)|=\frac{4^{i-1}(3n-5)+2}{3}\,|\Gix|$.
\end{lem}

\begin{proof}
We use induction on $i\geq 1$ to prove the first part of this lemma. If $i=1$, then $g\in X$ and 
$$\G(g)=\{(g,1,g_1),\,(g_1,1,g)\,|\;g_1\in X\setminus\{g\}\}\,;$$
whence $|\G(g)|=2n-2$, which is the desired value for $i=1$. 

Let now $i\geq 2$ and $g\in\G_i(X)$, and assume that 
$$|\G(g_1)|=(2^{2j-1}(3n-5)+4)/3$$
for all $g_1\in\G_j(X)$ with $1\leq j<i$. Consider the following two sets:
$$\G(g,l)=\{g_1\in\Gx\,|\; g\in\{g_1^l,g_1^r\}\mbox{ and }g_1^c=g^l\}$$
and 
$$\G(g,r)=\{g_1\in\Gx\,|\; g\in\{g_1^l,g_1^r\}\mbox{ and }g_1^c=g^r\}\,.$$
Then $\G(g)$ is the disjoint union of $\G(g,l)$ and $\G(g,r)$ (note that this conclusion is only true since $i\geq 2$; it fails for $i=1$ as $g^l=1=g^r$ in that case). Further, $|\G(g,l)|=2(|\G(g^l)|-1)$ and $|\G(g,r)|=2(|\G(g^r)|-1)$.  Thus
$$|\G(g)|=4(|\G(g^l)|-1)=4\left(\frac{2^{2i-3}(3n-5)+1}{3}\right)=\frac{2^{2i-1}(3n-5)+4}{3}\,.$$
We have shown the first equality.

For the second equality, begin by observing that $|\G_2(X)|=n(n-1)$ since the elements of $\G_2(X)$ are the triples $(g,1,g_1)$ with $g$ and $g_1$ distinct elements of $X$. Thus, the equality given is verified for $i=1$. Let now $i\geq 2$. Note that each element of $\G_{i+1}(X)$ belongs to precisely two sets of the form $\G(g)$, with $g\in \Gix$. Thus,
$$|G_{i+1}(X)|= \frac{1}{2}|\G(g)|\, |\Gix| = \frac{4^{i-1}(3n-5)+2}{3}\,|\Gix|
$$
for any $i\geq 2$.
\end{proof}

Let $g\in\Gix$ for some $i\geq 1$ and let $A$ be a subset of $k$ distinct elements of $\G(g)$. In what follows, we prove a sequence of results that culminates with Proposition \ref{r61}, where we conclude that $\ftx$ has a subsemigroup isomorphic to $\ft_k^1$ and with identity element $[g]$. So, up to Proposition \ref{r61}, $g$ will be a fixed element of $\Gix$ and $A$ will be a fixed subset $\{g_1,\cdots,g_k\}$ of $k$ elements of $\G(g)$. Then, no two elements of $[A]$ are $\Dc$-related. 

Set $A_0=\{g\}$, $A_1=A$, and $A_2=\{h\in\G_{i+2}(X)\,|\;h^l,h^r\in A_1\mbox{ and }h^c=g\}$. Set also
$$A_j=\{h\in \G_{i+j}(X)\,|\; h^l,h^r\in A_{j-1}\}$$
recursively for $j\geq 3$. Note that, if $h\in A_j$ with $j\geq 3$, then $h^l\in A_{j-1}$ and $h^c\in\{h^{l^2},h^{lr}\}\subseteq A_{j-2}$. Let $\ol{A}=\cup_{j\geq 0} A_j$ and 
$$\La_A(X)=\{u\in\lx\,|\;\si(u)=g=\tau(u)\mbox{ and }\c(u)\subseteq \ol{A}\}\,.$$

\begin{lem}
If $u\in\La_A(X)$ and $u\to v$, then $v$ and $\be_2(u)$ belong to $\La_A(X)$.
\end{lem}

\begin{proof}
Let $u=gh_1\cdots h_ng\in\La_A(X)$ and assume that $u\xr{h_j} v$ for a river $h_j$ of $u$. Then $1<j<n$ since $c(u)\subseteq \ol{A}$. If $h_{j-1}=h_{j+1}$, then $v=gh_1\cdots h_{j-1}h_{j+2}\cdots h_ng$ and, obviously, $v\in\La_A(X)$. So, assume that $h_{j-1}\neq h_{j+1}$. Then $v=gh_1\cdots h_{j-1}h_j'h_{j+1}\cdots h_ng$ for $h_j'=(h_{j+1},h_j,h_{j-1})$. If $h_j=g$, then $h_{j-1},h_{j+1}\in A_1$ and $h_j'\in A_2$; whence $v\in\La_A(X)$. If $h_j\neq g$, then $h_{j-1},h_{j+1}\in A_{j'}$ for some $j'\geq 2$, and so $h_j'\in A_{j'+1}$; whence $v\in\La_A(X)$ also. We have proved that $v\in\La_A(X)$. Now, $\be_2(u)$ also belongs to $\La_A(X)$ since $u\xr{*}\be_2(u)$.
\end{proof}

Let $u\in\La_A(X)$. By Lemma \ref{valelem}.$(iii)$, 
$$u\ap\la_l(g)*u*\la_r(g)\ap\la_l(g)*\be_2(u)*\la_r(g)\in\mxm\,,$$ 
and $\la_l(g)*\be_2(u)*\la_r(g)$ is the canonical form of $u$. Hence, $u\ap v$ for another $v\in\La_A(X)$ \iff\ $\be_2(u)=\be_2(v)$. Consequently, $[u]$ has a unique representative in the set $\La_A(X)\cap(\ltp\cup\{g\})$, namely $\be_2(u)$. Let $T$ be the subset 
$$T=\{[u]\in\ftx\,|\;u\in\La_A(X)\}$$ 
of $\ftx$. Let also $B_j=A_j$ if $j$ even, or let $B_j=\{ghg\,|\;h\in A_j\}$ otherwise; and set $\ol{B}=\cup_{j\geq 0} B_j$

\begin{prop}\label{p63}
$T$ is a regular subsemigroup of $\ftx$, it is generated by the set of idempotents $[\ol{B}]$, and $[g]$ is an identity element for $T$.
\end{prop}

\begin{proof}
Note that $u*v\in\La_A(X)$ if $u,v\in\La_A(X)$. Thus $T$ is a subsemigroup of $\ftx$. In fact, since $\cevl{\be_2(u)}\in\La_A(X)$ for any $u\in\La_A(X)$, $T$ is a regular subsemigroup by Proposition \ref{reg}. Clearly, $[g]$ is also an identity element for $T$: $g*u=u=u*g$ for any $u\in\La_A(X)$.

Let $h\in A_j$ for some $j\geq 2$ and set
$$u=h^{c^{j_1}}h^{c^{j_1-1}l}h^{c^{j_1-1}}\cdots h^ch^lhh^rh^c\cdots h^{c^{j_1-1}}h^{c^{j_1-1}r}h^{c^{j_1}}\,,$$
where $j_1$ is such that $j\in\{2j_1,2j_1+1\}$. Note that $h^l,h^r\in A_{j-1}$, $h^c\in A_{j-2}$, and $h^ch^lh\ap h\ap hh^rh^c$. Thus, recursively, we conclude that $h\ap u$, $\c(u)\subseteq\ol{A}$, and either $h^{c^{j_1}}=g$ if $j=2j_1$, or $h^{c^{j_1}}\in A_1$ if $j=2j_1+1$. If $j$ is even, then $u\in\La_A(X)$; whence $[h]=[u]\in T$. If $j$ is odd, then $ghg\ap gug\in\La_A(X)$; whence $[ghg]=[gug]\in T$. We can now conclude that $[\ol{B}]$ is a set of elements from $T$. Further, $[h]\leq [g]$ if $j$ is even. Thus, if $h_1h_2$ is a hill with $h_1,h_2\in\ol{A}$, then $h_1h_2\ap h_1gh_2$ (note that either $h_1$ or $h_2$ belong to some $A_j$ with $j$ even). Hence, if $gh_1\cdots h_ng\in\La_A(X)$, then
$$gh_1\cdots h_ng\ap (gh_1g)h_2(gh_3g)\cdots h_{n-1}(gh_ng)\in\langle[\ol{B} ]\rangle\,.$$
We have shown that $[\ol{B}]$ generates $T$.

Finally, let us prove that all elements from $[\ol{B}]$ are idempotents. Clearly, the elements from $[B_j]$ are idempotents if $j$ even or $j=1$. We prove that all other elements from $[\ol{B}]$ are also idempotents by induction on $j$ odd. Let $ghg\in B_j$ for some odd number $j$ greater than 1 and assume that all elements from $[B_{j-2}]$ are idempotents. In particular, we are assuming that $[gh^cg]$ is an idempotent. Hence,
$$(ghg)^2\ap ghh^rh^cgh^ch^lhg\ap ghh^r(gh^cg)^2h^lhg\ap ghh^rgh^cgh^lhg\ap ghg\,.$$
Thus, all elements from $[\ol{B}]$ are idempotents. 
\end{proof}

\begin{lem}\label{sandB}
Let $h\in A_j$ for some $j\geq 2$. Then 
$$S([(gh^rg)h^c],[h^c(gh^lg)])=\{[h]\}\;\mbox{ and }\;S([h^r(gh^cg)],[(gh^cg)h^l])=\{[ghg]\}\,.$$
\end{lem}

\begin{proof}
Clearly, 
$$S([(gh^rg)h^c],[h^c(gh^lg)])=S([gh^rh^c],[h^ch^lg])=S([h^rh^c],[h^ch^l])=\{[h]\}$$ 
and 
$$\begin{array}{ll}
S([h^r(gh^cg)],[(gh^cg)h^l])\hspace*{-.2cm} & =S([h^rh^cg],[gh^ch^l]) \\ [.2cm]
& = [gh^ch^l]\,V([h^rh^cgh^ch^l])\,[h^rh^cg] \\ [.2cm]
& = [g]\,([h^ch^l]\,V([h^rh^ch^l])\,[h^rh^c])\,[g] \\ [.2cm] 
& = [g]\,S([h^rh^c],[h^ch^l])\,[g]=\{[ghg]\}\,.\vspace*{-.4cm}
\end{array}$$
\end{proof}\vspace*{.3cm}

An immediate consequence of the previous lemma is that any regular subsemigroup of $\ftx$ containing $[B_1]$, contains also $[\ol{B}\setminus\{g\}]$. Hence,

\begin{cor}
$T\setminus\{[g]\}$ is the smallest regular subsemigroup of $\ftx$ containing $[B_1]$, while $T$ is the smallest regular subsemigroup of $\ftx$ containing $[B_1\cup\{g\}]$.
\end{cor}

Let $X_1=\{e_1,\cdots,e_k\}$ be a set of $k$ distinct elements. Define $\phi_1:\G(X_1)\to\ol{A}$ as follows:
$$h\phi_1=\left\{\begin{array}{ll}
g & \mbox{ if } h=1\,; \\ [.1cm]
g_j & \mbox{ if } h=e_j\in X_1\,;\\ [.1cm]
((h^l)\phi_1,(h^c)\phi_1,(h^r)\phi_1) & \mbox{ if } h\in\G_j(X_1) \mbox{ for some } j\geq 2\,.
\end{array}\right.$$

\begin{lem}
The mapping $\phi_1$ is a well defined bijection.
\end{lem}

\begin{proof}
We begin by showing inductively on $j$ that not only $\phi_1$ is well defined, but also that $(\G_j(X_1))\phi_1\subseteq A_j$. Clearly $1\phi_1=g\in A_0$ and $h\phi_1\in A_1$ for any $h\in X_1$. If $h\in\G_2(X_1)$, then $h=(e_{j_1},1,e_{j_2})$ for some $e_{j_1},e_{j_2}\in X_1$ such that $j_1\neq j_2$. Hence, $g_{j_1}$ and $g_{j_2}$ are distinct elements of $A_1$, $(g_{j_1},g,g_{j_2})\in A_2$, and $h\phi_1$ is defined as $(g_{j_1},g,g_{j_2})$. So, $h\phi_1$ is well defined and it belongs to $A_2$.

Let now $h\in\G_j(X_1)$ for some $j>2$, and assume that $h_1\phi_1$ is well defined and it belongs to $A_{j_1}$, for any $h_1\in\G_{j_1}(X_1)$ with $j_1<j$. In particular, $(h^l)\phi_1$, $(h^r)\phi_1$ and $(h^c)\phi_1$ are well defined, $(h^l)\phi_1,\,(h^r)\phi_1\in A_{j-1}$, and $(h^c)\phi_1\in A_{j-2}$. Further, by definition of $\phi_1$,
$$(h^l)\phi_1=((h^{l^2})\phi_1,(h^{lc})\phi_1,(h^{lr})\phi_1)$$
and 
$$(h^r)\phi_1=((h^{rl})\phi_1,(h^{rc})\phi_1,(h^{r^2})\phi_1)\,.$$
Since $h^c\in\{h^{l^2},h^{lr}\}\cap\{h^{rl},h^{r^2}\}$, we conclude that
$$((h^l)\phi_1,(h^c)\phi_1,(h^r)\phi_1)\in A_j\,.$$
Thus $h\phi_1$ is well defined and it belongs to $A_j$.

Now that we have shown that $\phi_1$ is well defined, let us prove that $\phi_1$ is a bijection. Clearly, $\phi_{1|\G_0(X_1)}$ is a bijection from $\G_0(X_1)=\{1\}$ onto $A_0=\{g\}$, while $\phi_{1|\G_1(X_1)}$ is a bijection from $\G_1(X_1)=X_1$ onto $A_1=A$. Let $j\geq 2$ and assume that for all $j_1<j$, $\phi_{1|\G_{j_1}(X_1)}$ is a bijection from $\G_{j_1}(X_1)$ onto $A_{j_1}$. 

If $\phi_1(h_1)=\phi_1(h_2)\in A_j$, then $h_1,h_2\in\G_j(X_1)$, and so $h_1^l,h_2^l,h_1^r,h_2^r\in\G_{j-1}(X_1)$ and $h_1^c,h_2^c\in \G_{j-2}(X_1)$. Further,  
$$(h_1^l)\phi_1=(h_2^l)\phi_1,\;\;(h_1^c)\phi_1=(h_2^c)\phi_1\;\;\mbox{ and }\;\;(h_1^r)\phi_1=(h_2^r)\phi_1\,.$$
Consequently, by the induction hypothesis, $h_1^l=h_2^l$, $h_1^c=h_2^c$ and $h_1^r=h_2^r$; whence $h_1=h_2$. We have shown that $\phi_{1|\G_j(X_1)}$ is a one-to-one mapping.

Let $h\in A_j$. By the induction hypothesis, there exist $h_1,h_2\in\G_{j-1}(X_1)$ and $h_3\in\G_{j-2}(X_1)$ such that
$$h_1\phi_1=h^l,\;\; h_3\phi_1=h^c\;\;\mbox{ and }\;\;h_2\phi_1=h^r\,.$$
Hence, $h_3\phi_1\in\{(h_1^l)\phi_1,(h_1^r)\phi_1\}\cap\{(h_2^l)\phi_1,(h_2^r)\phi_1\}$. But since $\phi_{1|\G_{j-2}(X_1)}$ is one-to-one, then 
$$h_3\in\{h_1^l,h_1^r\}\cap\{h_2^l,h_2^r\}\,.$$
Therefore, $(h_1,h_3,h_2)\in\G_j(X_1)$ and $(h_1,h_3,h_2)\phi_1=h$. We can now conclude that $\phi_{1|\G_j(X_1)}$ is a bijection from $\G_j(X_1)$ onto $A_j$, as wanted.
\end{proof}

Next, consider the mapping $\phi:\G(X_1)\to\ftx$ defined by 
$$h\phi=\left\{\begin{array}{ll}
[h\phi_1] & \mbox{ if }\up(h) \mbox{ is even;} \\ [.2cm]
[g(h\phi_1)g] & \mbox{ if }\up(h) \mbox{ is odd.}
\end{array}\right.$$

\begin{lem}
Then $\phi$ is a skeleton mapping whose image is $[\ol{B}]$. Thus $[\ol{B}\setminus\{g\}]$ is a skeleton of $\ftx$ (with respect to $X_1$).
\end{lem}

\begin{proof}
By the previous lemma, $\phi$ is a well defined mapping whose image is $[\ol{B}]$ and $\phi_{|X_1}$ is a one-to-one mapping. By Proposition \ref{p63}, $[\ol{B}]$ is a set of idempotents of $\ftx$ and $[g][h]=[h]=[h][g]$ for any $h\in\ol{B}$. Hence $\phi$ is a mapping from $\G(X_1)$ to $E(\ftx)$ satisfying conditions $(i)$ and $(ii)$ of the definition of skeleton mapping. Finally, By Lemma \ref{sandB}, $\phi$ also satisfies the condition $(iii)$ of the definition of skeleton mapping. Therefore, $\phi$ is a skeleton mapping and $[\ol{B}\setminus\{g\}]$ is a skeleton of $\ftx$.
\end{proof}

\begin{prop}\label{r61}
$T$ is isomorphic to $\ft^1(X_1)$.
\end{prop}

\begin{proof}
Using Proposition \ref{r52}, let $\varphi:\ft^1(X_1)\to\ftxm$ be the unique homomorphism extending the skeleton mapping $\phi$. Then $(\ft^1(X_1))\varphi=T$ by Proposition \ref{p63}.

Let $u=1h_1\cdots h_{2j-1}1\in\M(X_1)$. Note that $1\phi=[g]$ and 
$$h_{j_1}\phi=\left\{\begin{array}{ll} 
[g(h_{j_1}\phi_1)g] & \mbox{ if } j_1 \mbox{ odd;} \\ [.2cm] 
[h_{j_1}\phi_1] & \mbox{ if } j_1 \mbox{ even.}
\end{array}\right.$$ 
Since $g(h_{j_1}\phi_1)\ap h_{j_1}\phi_1\ap (h_{j_1}\phi_1)g$ if $j_1$ even, then
$$[u]\varphi =(1\phi)(h_1\phi)\cdots (h_{2j-1}\phi)(1\phi)=[\al(u)]\,,$$
where $\al(u)=g(h_1\phi_1)\cdots(h_{2j-1}\phi_1)g\in\La_A(X)\cap(\ltp\cup\{g\})$. Thus, if $u$ and $v$ are two mountains of $\ft^1(X_1)$ such that $[u]\varphi=[v]\varphi$, then $\al(u)=\al(v)$ since $[u]\phi$ and $[v]\phi$ have a unique representative in the set $\La_A(X)\cap(\ltp\cup\{g\})$. Consequently, $u=v$ because $\phi_1$ is a bijection from $\G(X_1)$ onto $\ol{A}$. We have shown that $\varphi$ is an isomorphism from $\ft^1(X_1)$ onto $T$; whence $T$ is isomorphic to $\ft^1(X_1)$.
\end{proof}

We reinforce that we have shown that if $g\in\Gx$ and $\{h_1,\cdots,h_k\}\subseteq\G(g)$, then $\ftx$ has a regular subsemigroup weakly generated by 
$$\{[gh_1g],\cdots, [gh_kg]\}$$ 
isomorphic to $\ft_k$. By Lemma \ref{sizeG}, we conclude also that $\ftx$ contains infinite many distinct copies of $\ft_k$, for any $k\in\mathbb{N}$, if $|X|\geq 2$. The next corollary combines Lemma \ref{sizeG} with the previous proposition.

\begin{cor}
Let $X$ and $Y$ be two sets such that $|X|=n$ and $|Y|=m$ with $m>n>1$. 
\begin{itemize}
\item[(i)] If $m\leq 2n-2$ and $g\in X$, then there exists an embedding $\varphi:\ft^1(Y)\to\ft(X)$ such that $[1]\varphi=[g]$ and, for each $y\in Y$, $[y]\varphi=[ghg]\in\Dcc_{[h]}$ for some $h\in\G_2(X)$.
\item[(ii)] If $i\geq\max\left\{2,\,\frac{1}{2}\log_2\left(\frac{6m-8}{3n-5}\right)\right\}$ and $g\in\G_i(X)$, then there exists an embedding $\varphi:\ft^1(Y)\to\ft(X)$ such that $[1]\varphi=[g]$ and, for each $y\in Y$, $[y]\varphi=[ghg]\in\Dcc_{[h]}$ for some $h\in\G_{i+1}(X)$.
\end{itemize}
\end{cor}

\begin{proof}
$(i)$. Assume that $n<m\leq 2n-2$ and let $g\in X$. Note that $\G(g)$ has $2n-2$ elements by Lemma \ref{sizeG}. So, let $A=\{h_1,\cdots,h_m\}$ be a set of $m$ elements of $\G(g)$. By the previous results, $\ftx$ has the smallest regular subsemigroup $T$ containing $\{[g]\}\cup\{[ghg]\,|\;h\in A\}$, and $T$ is isomorphic to $\ft^1(Y)$ under an isomorphism $\varphi:\ft^1(Y)\to\ftx$ such that $[1]\varphi=[g]$ and, for each $y\in Y$, $[y]\varphi=[ghg]\in\Dcc_{[h]}$ for some $h\in A$.

$(ii)$. Assume that $i\geq\max\left\{2,\,\frac{1}{2}\log_2\left(\frac{6m-8}{3n-5}\right)\right\}$ and let $g\in\G_i(X)$. The proof of $(ii)$ is similar to the proof of $(i)$ and we just need to show that $\G(g)$ has, at least, $m$ elements. But,
$$i\geq \frac{1}{2}\log_2\left(\frac{6m-8}{3n-5}\right)\;\Longleftrightarrow\;\frac{2^{2i-1}(3n-5)+4}{3}\geq m\,.$$
Hence, $|\G(g)|\geq m$ and the embedding $\varphi$ described in $(ii)$ exists.
\end{proof}

\begin{cor}
$\ftt$ has a copy of $\mft$ as a subsemigroup for all $m\geq 2$. Furthermore, all regular semigroups weakly generated by a finite set of idempotents (including all finitely idempotent generated regular semigroups) strongly divide $\ftt$.
\end{cor}

\begin{proof}
The first part of this corollary is just a particular case of the previous corollary. The second part now follows from Proposition \ref{fiwig}.
\end{proof}

Since every finite semigroup can be embedded into a finite idempotent generated regular semigroup \cite{howie66}, we obtain also the following result:

\begin{cor}
Every finite semigroup divides $\ftt$.
\end{cor}

\section{The principal factors of $\ftx$}

In this last section, we are going to work with the model $\mx$ of $\ftx$ instead of $\ftx$. Hence, in this section, $\Dcc_g$ denotes the $\Dc$-class of $\be(g)$ in $\mx$ for any $g\in\Gx\setminus\{1\}$. Then $\Dcc_g=\{u\in \mx\,|\;\ka(u)=g\}$. Let $F_g$ denote the principal factor of $\be(g)$ in $\mx$, that is, $F_g$ is the semigroup $\Dcc_g\cup\{0\}$ with a zero element $0$ and product $\cdot$ defined by
$$u\cdot v=\left\{\begin{array}{ll}
u\odot v & \mbox{ if } u\odot v\in \Dcc_g \\ [.2cm]
0 & \mbox{ otherwise}\;,
\end{array}\right.$$
for any mountains $u,v\in \Dcc_g$. In this section, we analyze and compare the structure of these principal factors. We begin by studying the idempotents of $\Dcc_g$.

\subsection{Idempotents of $\Dcc_g$} 

Let $u=g_n\cdots g_0\in\lom$ with $g_0\neq 1$, and let $q_i\in\{l,r\}$ be such that $g_{i-1}=g_i^{q_i}$, for $i\in\{1,\cdots,n\}$. Define $\al(u)=q_n\cdots q_1\in\{l,r\}^+$. If $g_0\in\G_1(X)$, define also $\al(u1)=\al(u)$.  

Let $S_i=\{s_{i,1},\cdots, s_{i,2^{i-1}}\}$ be the list of all $2^{i-1}$ words of length $i-1$ from $\{l,r\}^+$, listed by increasing order with respect to the lexicographic order $\preceq$. For each $g\in\Gix$ with $i\geq 1$, let $\La_1^-(X,g)$ be the set of all downhills from $g$ to $1$. Clearly, $\alpha$ defines a bijection from $\La_1^-(X,g)$ onto $S_i$. Consider
$$\La_1^-(X,g)=\{v_{g,1},\cdots, v_{g,2^{i-1}}\}$$
such that $\al(v_{g,k})=s_{i,k}$.

We now use the following notation to refer to the $\Lc$-classes of $\Dcc_g$: $\Lcc_{g,k}$ is the $\Lc$-class of $\Dcc_g$ constituted by all mountains $u$ such that $\la_r(u)=v_{g,k}$, or equivalently, constituted by all mountains $u\in\Dcc_g$ such that $\al(\la_r(u))=s_{i,k}$. We introduce also a similar notation for the $\Rc$-classes of $\Dcc_g$: $\Rcc_{g,k}$ is the $\Rc$-class of $\Dcc_g$ constituted by all mountains $u$ such that $\cevl{\la_l(u)}=v_{g,k}$. We define the \emph{canonical ``egg-box'' representation} of $\Dcc_g$ as its ``egg-box'' representation where the $k^{th}$ column is the $\Lc$-class $\Lcc_{g,k}$ and the $k^{th}$ row is the $\Rc$-class $\Rcc_{g,k}$. We denote this representation by $E_g$.

Note that $\cevl{v_{g,j}}*v_{g,k}$ is the only mountain in the $\Hc$-class $\Hcc_{g,j,k}=\Rcc_{g,j}\cap\Lcc_{g,k}$ of $E_g$. Thus, all $\Hc$-classes $\Hcc_{g,k,k}$ in the main diagonal of $E_g$ are trivial groups, that is, they are constituted by a single idempotent. Our intention in this subsection is to prove that there are no idempotents below the main diagonal of $E_g$.

Let $u=g_n\cdots g_0h_1\cdots h_m$ be a valley with river $g_0$. Since $\be_2(u)$ has no rivers, $\be_2(u)\in\ltp\cup\lo\cup\Gx$, $\si(\be_2(u))=g_n$ and $\tau(\be_2(u))=h_m$. Further, if $\be_2(u)\in\lop$ then $n<m$; if $\be_2(u)\in\lom$ then $n>m$; and if $\be_2(u)\in\Gx$ then $n=m$ and $g_n=h_m$. The next result give us more details about the words $\be_2(u)$, but we need to introduce the following notions first. 

Note that any hill from $g_n$ to $h_m$ (if it exists) has length $|n-m|+1$. If $n>m$, then an \emph{$u$-hill} is a downhill 
$$g_n\cdots g_kg_k^l\cdots g_k^{l^{k-m}}$$
for some $m\leq k\leq n$ such that 
$$g_k^{l^{k-m}}=h_m,\;\; g_k^l\neq g_{k-1} \;(\mbox{if } k\neq m),\;\mbox{ and }\; g_k^{l^{i-2}c}\neq g_k^{l^i} \mbox{ for } 2\leq i\leq k-m\,. $$ 
Thus, for $n>m$, an $u$-hill is a downhill from $g_n$ to $h_m$ with some extra properties, and a particular case worth noticing occurs when $g_m=h_m$, namely $g_n\cdots g_m$.  

Now, if $n<m$,  we then define an \emph{$u$-hill} dually, that is, as an uphill 
$$h_k^{r^{k-n}}\cdots h_k^rh_k\cdots h_m$$
for some  $n\leq k\leq m$ such that 
$$h_k^{r^{k-n}}=g_n,\;\; h_k^r\neq h_{k-1} \;(\mbox{if } k\neq n),\;\mbox{ and }\; h_k^{r^{i-2}c}\neq h_k^{r^i} \mbox{ for } 2\leq i\leq k-n\,. $$ 
We also define a \emph{$(p,q)$-landscape} as a landscape
$$g^{r^p}g^{r^{p-1}}\cdots g^rgg^l\cdots g^{l^{q-1}}g^{l^q}\in\ltp\,,$$
for some $g\in\Gx$ such that $g^{r^{i-2}c}\neq g^{r^i}$ for $2\leq i\leq p$ and $g^{l^{j-2}c}\neq g^{l^j}$ for $2\leq j\leq q$.

\begin{lem}\label{lem71}
Let $u=g_n\cdots g_0h_1\cdots h_m$ be a valley with river $g_0$. Then $\be_2(u)$ is either a $(p,q)$-landscape with $p\leq m$ and $q\leq n$, an $u$-hill, or just $g_n$ with $n=m$ and $g_n=h_m$.
\end{lem}

\begin{proof}
We prove this result by induction on the length of the valley $u$. If $u=g_1g_0h_1$, then either $\be_2(u)=g_1$ if $g_1=h_1$, or $\be_2(u)=g_1gh_1$ for $g=(h_1,g_0,g_1)$ if $g_1\neq h_1$. Note that, in the latter case, $\be_2(u)$ is a $(1,1)$-landscape. 

Assume now that $\l(u)>3$ and that the result holds for any valley of length less than $\l(u)$. If $m=1$, then $n\geq 2$ and let $v$ be the valley $g_{n-1}\cdots g_0h_1$. We can apply the induction hypothesis to $v$ and assume that $\be_2(v)$ is either a $(1,q)$-landscape with $q\leq n-1$, a $v$-hill, or just $g_1$ with $n=2$ and $g_1=h_1$. In the latter two cases, $\be_2(u)=g_n\be_2(v)$ and $\be_2(u)$ is an $u$-hill. In the former case, note that $\be_2(v)$ is, in fact, a $(1,n-1)$-landscape, that is,
$$\be_2(v)=g^rgg^l\cdots g^{l^{n-1}}$$
for $g^r=g_{n-1}\neq g^l$, $g^{l^{n-1}}=h_1$ and $g^{l^{i-2}c}\neq g^{l^i}$ for $2\leq i\leq n-1$. If $g=g_n$, then $\be_2(u)=gg^l\cdots g^{l^{n-1}}$ and $\be_2(u)$ is an $u$-hill. If $g\neq g_n$, then
$$\be_2(u)=g_nhgg^l\cdots g^{l^{n-1}}$$
for $h=(g,g_{n-1},g_n)$. Hence $g=h^l$, $g_n=h^r$, $h^c=g_{n-1}\neq g^l=h^{l^2}$, and $\be_2(u)$ is a $(1,n)$-landscape. Note that we have proved that this result holds for valleys $g_n\cdots g_0h_1$ with river $g_0$. 

Let us consider now the general case with $m>1$, and let $v$ be the valley 
$$v=g_n\cdots g_0h_1\cdots h_{m-1}$$
with river $g_0$. Applying the induction hypothesis to $v$, we can assume that $\be_2(v)$ is either a $(p,q)$-landscape with $p\leq m-1$ and $q\leq n$, a $v$-hill, or just $g_n$ with $n=m-1$ and $g_n=h_{m-1}$. Note that $\be_2(u)=\be_2(v)h_m$ if $\be_2(v)=g_n$ or if $\be_2(v)$ is a $v$-hill with $n<m-1$. In both cases, $\be_2(u)$ becomes an $u$-hill. 

Assume that $\be_2(v)$ is a $v$-hill with $n>m-1$. Then $\be_2(u)=\be_2(\be_2(v)h_m)$ and $\be_2(v)h_m$ is a valley with river $h_{m-1}$. We can apply now the case $m=1$ to the valley $\be_2(v)h_m$ and conclude that $\be_2(u)=\be_2(\be_2(v)h_m)$ is either a $(1,n-m+1)$-landscape, a $\be_2(v)h_m$-hill, or just $g_n$ with $n=m$ and $g_n=h_m$. If $\be_2(u)$ is a $\be_2(v)h_m$-hill, we need to consider two cases:
$$\be_2(v)=g_n\cdots g_{m-1}\;\;\mbox{ with }\;\; g_{m-1}=h_{m-1}\,,$$
or
$$\be_2(v)=g_n\cdots g_kg_k^l\cdots g_k^{l^{k-m+1}}$$
for some $m\leq k\leq n$ such that  $g_k^{l^{k-m+1}}=h_{m-1}$, $g_k^l\neq g_{k-1}$ and $g_k^{l^{i-2}c}\neq g_k^{l^i}$ for $2\leq i\leq k-m+1$.
In the former case, $\be_2(u)$ is clearly an $u$-hill also. In the latter case, note that 
$$\be_2(u)=g_n\cdots g_{k_1}g_{k_1}^l\cdots g_{k_1}^{l^{k_1-m}}$$
for some $k\leq k_1\leq n$ such that  $g_{k_1}^{l^{k_1-m}}=h_m$, $g_{k_1}^l\neq g_{k_1-1}$ and $g_{k_1}^{l^{i-2}c}\geq g_{k_1}^{l^i}$ for $2\leq i\leq k_1-m$. Thus $\be_2(u)$ is again an $u$-hill. We have shown that $\be_2(u)$ is either a $(1,n-m+1)$-landscape, an $u$-hill, or just $g_n$ with $n=m$ and $g_n=h_m$ if $\be_2(v)$ is a $v$-hill with $n>m-1$.

Finally, assume that $\be_2(v)$ is an $(p,q)$-landscape with $p\leq m-1$ and $q\leq n$, that is,
$$\be_2(v)=g^{r^p}\cdots g^rgg^l\cdots g^{l^q}\in\ltp\,,$$
for some $g\in\Gx$,  $p\leq m-1$ and $q\leq n$ such that $g^{r^{i-2}c}\neq g^{r^i}$ for $2\leq i\leq p$ and $g^{l^{j-2}c}\neq g^{l^j}$ for $2\leq j\leq q$. Note that $g^{r^p}=g_n$  and $v_1=gg^l\cdots g^{l^q}h_m$ is a valley with river $g^{l^q}=h_{m-1}$. If $q=1$ and $h_m=g$, then $\be_2(u)=g^{r^p}\cdots g^rg$ and $\be_2(u)$ is an $u$-hill. If $q>1$ and $h_m=g^{l^{q-1}}$, then $\be_2(u)=g^{r^p}\cdots g^r\be_2(v_1)=g^{r^p}\cdots g^rgg^l\cdots g^{l^{q-1}}$ and $\be_2(u)$ is a $(p,q-1)$-landscape. If $h_m\neq g^{l^{q-1}}$ (or $h_m\neq g$ if $q=1$), then observe that
$$\be_2(v_1)=ghh^l\cdots h^{l^q}\,,$$
where $h^{l^q}=h_m$ and $h^{l^i}=(h^{l^{i+1}},g^{l^{i+1}},g^{l^i})$ for $i\in\{0,\cdots,q-1\}$. Therefore,
$$\be_2(u)=h^{r^{p+1}}\cdots h^rhh^l\cdots h^{l^q}$$
and $\be_2(u)$ is a $(p+1,q)$-landscape with $p+1\leq m$ and $q\leq n$.
\end{proof}

\begin{lem}
Let $u=g_0g_1\cdots g_{2n-1}g_{2n}$ be a gorge with $g_0=g=g_{2n}$. Then $\al(w)\preceq\al(v)$ for $v=g_0\cdots g_n$ and $w=g_{2n}\cdots g_n$.
\end{lem}

\begin{proof}
Fix $i\in\{1,\cdots,n-1\}$ and let $u_i$ be the prefix $g_0\cdots g_{n+i}$ of $u$.  Thus each $u_i$ is a valley with river $g_n$ and
$$g=\be_2(u)=\be_2(\be_2(u_i)g_{n+i+1}\cdots g_{2n})\,.$$ 
By Lemma \ref{lem71}, $\be_2(u_i)$ must be an $u_i$-hill. Therefore, 
$$v_i=\be_2(u_i)g_{n+i-1}\cdots g_n$$
is a downhill from $g$ to $g_n$. Note also that $v_{n-1}=w$ since $\be_2(u_{n-1})=g_0g_{2n-1}=g_{2n}g_{2n-1}$. Set $v_0=v$, another downhill from $g$ to $g_n$. Since all words $\al(v_i)$ have length $n$, this lemma becomes proved once we show that $\al(v_i)\preceq\al(v_{i-1})$ for all $1\leq i\leq n-1$. 

Note that $v_1=\be_2(u_1)g_n=\be_2(v_0g_{n+1})g_n$. If $g_{n+1}=g_{n-1}$, then $v_1=v_0$. If $g_{n+1}\neq g_{n-1}$, then 
$$\be_2(u_1)=g_0\cdots g_kg_k^l\cdots g_k^{l^{n-1-k}}\,,$$ 
for some $0\leq k< n-1$ such that $g_k^{l^{n-1-k}}=g_{n+1}$, $g_k^l\neq g_{k+1}$ and $g_k^{l^{i-2}c}\neq g_k^{l^i}$ for $2\leq i\leq n-1-k$, by Lemma \ref{lem71}. Hence, $g_{k+1}=g_k^r$. If $s=\al (g_0\cdots g_k)$, then $sl$ is a prefix of $\al(v_1)$, while $sr$ is a prefix of $\al(v_0)$. Thus $\al(v_1)\preceq\al(v_0)$ as wanted. 

Assume now that $1<i<n$. Then $\be_2(u_i)=\be_2(\be_2(u_{i-1})g_{n+i})$ is an $u_i$-hill. But then $\be_2(u_i)$ is also a $\be_2(u_{i-1})g_{n+i}$-hill, again by Lemma \ref{lem71}. Similarly to the previous case, we can now deduce that $\al(\be_2(u_i)g_{n+i-1})\preceq\al(\be_2(u_{i-1}))$; whence $\al(v_i)\preceq\al(v_{i-1})$.
\end{proof}

We have now all the ingredients necessary to prove that $E_g$ has no idempotents below the main diagonal.

\begin{prop}\label{idEg}
The canonical ``egg-box'' representation $E_g$ of $\Dcc_g$ has no idempotents below the main diagonal or, in other words, there are no idempotents $u$ in $\Dcc_g$ such that $\al(\la_r(u))\prec\al(\cevl{\la_l(u)})$.
\end{prop}

\begin{proof}
If $u=\la_l(u)*\la_r(u)$ is an idempotent of $E_g$, then $\la_r(u)*\la_l(u)$ is a gorge. By the previous lemma, $\al(\cevl{\la_l(u)})\preceq\al(\la_r(u))$ and $u$ is not below the main diagonal of $E_g$.
\end{proof}

About the $\Hc$-classes above the main diagonal in $E_g$, it is not immediately which ones contain idempotents. However, since $\ftx$ is idempotent generated, we can say something more in this regard. Note that if $S$ idempotent generated and $e$ and $f$ are two $\Dc$-related idempotents, then there exists a sequence $e_0,e_1,\cdots,e_n$ of idempotents such that 
$$e=e_0\Rc e_1\Lc e_2\Rc\cdots\Rc e_{n-1}\Lc e_n=f\,.$$

Let $\Gamma_g=(V,E)$ be the (non-oriented) graph with set of vertices 
$$V=L_1^-(X,g)=\{v_{g,1},\cdots,v_{g,2^{i-1}}\}$$
of all downhills from $g$ to $1$, and set of edges
$$E=\{(v_{g,j},v_{g,k})\,|\;j<k \mbox{ and } \cevl{v_{g,k}}*v_{g,j} \mbox{ is an idempotent}\}\,.$$
Thus, the vertices of $\Gamma_g$ can be seen as representing the idempotents $\cevl{v_{g,j}}*v_{g,j}$ in the main diagonal of $E_g$, while the edges represent the idempotents of $E_g$ above the main diagonal. The next result is now a trivial consequence of the observation made above about the $\Dc$-related idempotents of idempotent generated semigroups.

\begin{prop}
The graph $\Gamma_g$ is connected.
\end{prop}

We can say something also about the product of two elements of $\mx$. In the next result we prove that if $u\odot v\in E(\mx)$ for $u,v\in\mx$, then $u\Rc u\odot v$ or $v\Lc u\odot v$.

\begin{prop}\label{uvid}
Let $u,v\in\mx$. If $u\odot v\in E(\mx)$, then $u\Rc u\odot v$ or $v\Lc u\odot v$.
\end{prop}

\begin{proof}
Assume that $w=u\odot v\in\Dcc_g$ for some $g\in\Gx$ and that $u$ is not $\Rc$-related with $w$ and $v$ is not $\Lc$-related with $w$. Note that
$$w=\la_l(u)*\be_2(\la_r(u)*\la_l(v))*\la_r(v)\,.$$
By Corollary \ref{R} and Lemma \ref{lem71}, $\be_2(\la_r(u)*\la_l(v))$ must be a $(p,q)$-landscape since neither $u$ is $\Rc$-related with $w$, nor $v$ is $\Lc$-related with $w$. Consequently $g^rgg^l$ is a subword of $w$, $l$ is a prefix of $\al(\la_r(w))$ and $r$ is a prefix of $\al(\cevl{\la_l(w)})$. We have shown that
$$\al(\la_r(w))\prec\al(\cevl{\la_l(w)})\, ;$$
whence $w$ is not an idempotent.
\end{proof}

\subsection{Comparing $\Dcc_g$ with $\Dcc_{g^l}$, $\Dcc_{g^r}$ and $\Dcc_{g^c}$}

Let $g\in\Gix$ and let $s,t\in\{l,r\}^*$ of length $j$ and $k$, respectively, less than or equal to $i$. Note that $j$ or $k$ are $0$ if $s$ or $t$ are the empty word, respectively. For each $n\in\{1,\cdots, j\}$ and each $m\in\{1,\cdots, k\}$, let $s_n$ be the prefix of $s$ of length $n$ and $t_m$ be the prefix of $t$ of length $m$. Set
$$\Dcc_g^{s,t}=\{u\in M\,|\; g^{s_j}\cdots g^{s_1}gg^{t_1}\cdots g^{t_k} \mbox{ is a subword of } u\}\subseteq\Dcc_g\,.$$
Thus $\Dcc_g^{s,t}$ is the set of all mountains $u$ such that $s$ and $t$ are prefixes of $\al(\cevl{\la_l(u)})$ and $\al(\la_r(u))$, respectively.
Note further that $\Dcc_g^{l,l}$, $\Dcc_g^{l,r}$, $\Dcc_g^{r,l}$ and $\Dcc_g^{r,r}$ are a partition of $\Dcc_g$ into four parts if $i\geq 2$. 

Let $F_g^{s,t}$ denote the subsemigroup $\Dcc_g^{s,t}\cup\{0\}$ of $F_g$. The next result gathers information about $\Dcc_g^{r,l}$ and $F_g^{r,l}$ already obtained in previous results.

\begin{prop}
Let $g\in\Gix$ with $i\geq 2$ and $u,v\in\mx$.
\begin{itemize}
\item[$(i)$] If $u\odot v\in\Dcc_g$ but $u,v\not\in\Dcc_g$, then $u\odot v\in \Dcc_g^{r,l}$.
\item[$(ii)$] $F_g^{r,l}$ is a zero semigroup, that is, $u\odot v=0$ for all $u,v\in F_g^{r,l}$.
\end{itemize}
\end{prop}

\begin{proof}
Note that $(i)$ follows from the proof of Proposition \ref{uvid}, while $(ii)$ follows from Proposition \ref{idEg}.
\end{proof}

The next result is the key ingredient to prove that $F_g^{l,l}$ and $F_g^{r,r}$ are isomorphic to $F_{g^l}$ and $F_{g^r}$, respectively.

\begin{lem}\label{gorgell}
Let $u=gg_1\cdots g_{2n-1} g$ be a canyon with $g_1=g_{2n-1}$. Then $u$ is a gorge \iff\ $n=1$ or $v=g_1\cdots g_{2n-1}$ is a gorge.
\end{lem}

\begin{proof}
There are two cases to consider: $g_1=g_{2n-1}=g^l$ or $g_1=g_{2n-1}=g^r$. Since they are similar, we shall prove only the case $g_1=g_{2n-1}=g^l$. So, assume that $g_1=g_{2n-1}=g^l$. Without loss of generality, we can assume also that $n>1$. Further, it is obvious that $u$ is a gorge if $v$ is a gorge. Hence, we only need to prove that if $u=gg^lg_2\cdots g_{2n-2}g^lg$ is a gorge for some $n>1$, then $v=g^lg_2\cdots g_{2n-2}g^l$ is a gorge also. 

By Lemma \ref{lem71}, $\be_2(v)$ is either $g^l$ or a $(p,q)$-landscape. However, $\ep(\be_2(v))\subseteq\ep(g)$ since $g=\be_2(u)=\be_2(g\be_2(v)g)$. Note now that there are no $(p,q)$-landscapes $v_1$ form $g^l$ to $g^l$ with $\ep(v_1)\subseteq\ep(g)$; whence $\be_2(v)=g^l$ and $v$ is a gorge.
\end{proof}

We introduce now the following two mappings:
$$\begin{array}{rccc}
\varphi_g^l:\!\!\!&\Dcc_{g^l}\!\!&\to\!&\Dcc_g^{l,l}\\ &u&\mapsto& \la_l(u)\,g\,\la_r(u)
\end{array} \quad\;\mbox{ and }\;\quad
\begin{array}{rccc}\varphi_g^r:\!\!\!&\Dcc_{g^r}\!\!&\to\!&\Dcc_g^{r,r} \\ &u&\mapsto&\la_l(u)\,g\,\la_r(u)
\end{array}\;.$$
They are well defined since $\ka(u)=g^l$ if $u\in\Dcc_{g^l}$, and $\ka(u)=g^r$ if $u\in \Dcc_{g^r}$. We extend both $\varphi_g^l$ and $\varphi_g^r$ to mappings 
$$\varphi_g^l:F_{g^l}\to F_g^{l,l}\quad\mbox{ and }\quad \varphi_g^r:F_{g^r}\to F_g^{r,r}\,,$$
respectively, by setting $0\varphi_g^l=0=0\varphi_g^r$.

\begin{prop}\label{Dll}
The mappings $\varphi_g^l:F_{g^l}\to F_g^{l,l}$ and $\varphi_g^r:F_{g^r}\to F_g^{r,r}$ are isomorphisms. Thus, a mountain $u_1g^lgg^lu_2$ {\rm[}$u_1g^rgg^ru_2${\rm ]} is an idempotent \iff\ the mountain $u_1g^lu_2$ {\rm [}$u_1g^ru_2${\rm ]} is an idempotent too.
\end{prop}

\begin{proof}
We shall prove only that $\varphi_g^l$ is an isomorphism since the proof for $\varphi_g^r$ is similar and the second part of this proposition follows obviously from the first part. Note also that $\varphi_g^l$ is a bijection with inverse mapping 
$$u_1g^lgg^lu_2\mapsto u_1g^lu_2\,.$$ 
So, we only need to prove that $\varphi_g^l$ is a homomorphism.

Let $u,v\in \Dcc_{g^l}$ and let $u_1=u\varphi_g^l$ and $v_1=v\varphi_g^l$. Then $u_1=\la_l(u)\,g\,\la_r(u)$ and $v_1=\la_l(v)\,g\,\la_r(v)$. Note that $u\odot v\in \Dcc_{g^l}$ \iff\ the canyon $\la_r(u)*\la_l(v)$ is a gorge. By Lemma \ref{gorgell}, $\la_r(u)*\la_l(v)$ is a gorge \iff\ $g\,\la_r(u)*\la_l(v)\,g$ is a gorge too. Hence, $u\odot v\in D_{g^l}$ \iff\ $u_1\odot v_1\in D_g^{l,l}$. Consequently, if $u\odot v\not\in D_{g^l}$, then
$$(u\cdot v)\varphi_g^l=0\varphi_g^l=0=u_1\cdot v_1=(u\varphi_g^l)\cdot (v\varphi_g^l)\,.$$

Assume now that $u\odot v\in D_{g^l}$. Then $u\cdot v=u\odot v$ and $u_1\cdot v_1=u_1\odot v_1$. Since $u\odot v=\la_l(u)*\la_r(v)$ by Corollary \ref{R}, we conclude that
$$\begin{array}{ll}
(u\cdot v)\varphi_g^l\!\! & = \la_l(u)\,g\,\la_r(v)\ap \la_l(u)\,g\,\la_r(u)*\la_l(v)\,g\,\la_r(v) =u_1*v_1 \\ [.2cm]
&\ap u_1\odot v_1=(u\varphi_g^l)\cdot (v\varphi_g^l)\,.
\end{array}$$
We have shown that $\varphi_g^l$ is a homomorphism as wanted.
\end{proof}

We have seen so far that $\Dcc_g^{r,l}$ has no idempotents and that $\Dcc_g^{l,l}$ and $\Dcc_g^{r,r}$ are copies of $\Dcc_{g^l}$ and $\Dcc_{g^r}$, respectively. Let us look now to $\Dcc_g^{l,r}$, and assume that $g\in \Gix$ with $i\geq 3$. Note that $\Dcc_g^{l,r}$ is partitioned into the four sets $\Dcc_g^{l_c,r_c}$, $\Dcc_g^{l_o,r_c}$, $\Dcc_g^{l_c,r_o}$ and $\Dcc_g^{l_o,r_o}$. 

\begin{lem}\label{gorgercl}
Let $u=gg^rg^c\cdots g_{2n-2}g^lg$ and $v=g^lg^c\cdots g_{2n-2}g^l$ be two canyons. The following conditions are equivalent:
\begin{itemize}
\item[$(i)$] $u$ is a gorge.
\item[$(ii)$] $v$ is a gorge.
\item[$(iii)$] $gvg$ is a gorge.
\end{itemize}
\end{lem}

\begin{proof}
The equivalence between $(ii)$ and $(iii)$ follows from Lemma \ref{gorgell}. So, we only need to prove the equivalence between $(i)$ and $(ii)$. Assume first that $v$ is a gorge and let $v_1=g^c\cdots g_{2n-2}g^l$. Since $g^l=\be_2(v)=\be_2(g^l\be_2(v_1))$, then $\be_2(v_1)=g^cg^l$ by Lemma \ref{lem71}. Hence 
$$u\xr{*}gg^r\be_2(v_1)g= gg^rg^cg^lg\xr{g^c}gg^rgg^lg\xr{*} g\,,$$
and $u$ is a gorge.

Assume now that $u$ is a gorge. By Lemma \ref{lem71}, either $\be_2(v_1)=g^cg^l$ or $\be_2(v_1)$ is a $(p,q)$-landscape. Since $g=\be_2(u)=\be_2(gg^r\be_2(v_1)g)$, we also know that $\ep(\be_2(v_1))\subseteq\ep(g)$. However, observe that there is no $(p,q)$-landscape $w$ from $g^c$ to $g^l$ such that $\ep(w)\subseteq \ep(g)$. Consequently $\be_2(v_1)=g^cg^l$, and so $\be_2(v)=g^l$ and $v$ is a gorge.
\end{proof}

We have now the following corollary.

\begin{cor}\label{idemp23}
\begin{itemize}
\item[$(i)$] A mountain $v_1g^lgg^lg^cv_2\in\Dcc_g^{l,l_c}$ is an idempotent \iff\ the mountain $v_1g^lgg^rg^cv_2\in\Dcc_g^{l,r_c}$ is an idempotent.
\item[$(ii)$] A mountain $v_1g^cg^rgg^rv_2\in\Dcc_g^{r_c,r}$ is an idempotent \iff\ the mountain $v_1g^cg^lgg^rv_2\in\Dcc_g^{l_c,r}$ is an idempotent.
\end{itemize}
\end{cor}

\begin{proof}
The first statement follows from Lemma \ref{gorgercl} and Proposition \ref{idgorges}. The second statement is the dual of the first; whence it follows from the dual of Lemma \ref{gorgercl}.
\end{proof}

Let $\phi_g^l:F_g^{l,l_c}\to F_g^{l,r_c}$ be the mapping defined by 
$$0\phi_g^l=0\quad\mbox{ and }\quad (v_1g^lgg^lg^cv_2)\phi_g^l= v_1g^lgg^rg^cv_2\,,$$
for each mountain $v_1g^lgg^lg^cv_2\in\Dcc_g^{l,l_c}$. Analogously, consider the mapping $\phi_g^r:F_g^{r_c,r}\to F_g^{l_c,r}$ defined by
$$0\phi_g^r=0\quad\mbox{ and }\quad (v_1g^cg^rgg^rv_2)\phi_g^r= v_1g^cg^lgg^rv_2\,,$$
for each mountain $v_1g^cg^rgg^rv_2\in\Dcc_g^{r_c,r}$. In the next result we show that these two mappings are isomorphisms.

\begin{prop}\label{Dllc}
The mappings $\phi_g^l$ and $\phi_g^r$ are isomorphisms. Further, in $\mx$, $u\Rc (u\phi_g^l)$ for any mountain $u\in \Dcc_g^{l,l_c}$ and $v\Lc (v\phi_g^r)$ for any mountain $v\in\Dcc_g^{r_c,r}$.
\end{prop}

\begin{proof}
The proof that $\phi_g^l$ is an isomorphism is similar to the proof of Proposition \ref{Dll} using Lemma \ref{gorgercl} instead of Lemma \ref{gorgell}. We omit the details here. We can show that $\phi_g^r$ is a homomorphism too using a similar strategy based on the dual of Lemma \ref{gorgercl}. Thus both $\phi_g^l$ and $\phi_g^r$ are isomorphisms since, clearly, they are bijections. The second statement is just a consequence of Corollary \ref{R}.
\end{proof}

There are four other isomorphisms that can be obviously derived from restriction of $\phi_g^l$ and $\phi_g^r$. We list them in the following corollary:

\begin{cor}\label{Dlclc}
The following mappings are isomorphisms: 
$$\begin{array}{ll}
\phi_g^l|_{F_g^{l_c,l_c}}:F_g^{l_c,l_c}\to F_g^{l_c,r_c}\,;\qquad &
\phi_g^l|_{F_g^{l_o,l_c}}:F_g^{l_o,l_c}\to F_g^{l_o,r_c}\,; \\ [.4cm]
\phi_g^r|_{F_g^{r_c,r_c}}:F_g^{r_c,r_c}\to F_g^{l_c,r_c}\,;\qquad &
\phi_g^r|_{F_g^{r_c,r_o}}:F_g^{r_c,r_o}\to F_g^{l_c,r_o}\,.
\end{array}$$
Furthermore, the semigroups $F_g^{l_c,l_c}$, $F_g^{l_c,r_c}$ and $F_g^{r_c,r_c}$ are isomorphic to the completely 0-simple semigroup $F_{g^c}$.
\end{cor}

\begin{proof}
These four mappings are clearly isomorphisms. Thus $F_g^{l_c,l_c}$, $F_g^{l_c,r_c}$ and $F_g^{r_c,r_c}$ are isomorphic semigroups.  If $g^c=g^{l^2}$, then $\varphi_{g^l}^l:F_{g^c}\to F_{g^l}^{l,l}$ is an isomorphism by Proposition \ref{Dll}. Hence $\varphi_{g^l}^l\circ\varphi_g^l$ is a one-to-one homomorphism whose image is $F_g^{l_c,l_c}$. Therefore,
$$\varphi_{g^l}^l\circ\varphi_g^l:F_{g^c}\to F_g^{l_c,l_c}$$
is an isomorphism. If $g^c=g^{lr}$, then we can show similarly that $\varphi_{g^l}^r\circ\varphi_g^l:F_{g^c}\to F_g^{l_c,l_c}$ is an isomorphism.
\end{proof} 

We point out that the previous two results allow us to identify the idempotents of $\Dcc_g$ if we know the idempotents of $\Dcc_{g^l}$ and $\Dcc_{g^r}$, except for the idempotents in $\Dcc_g^{l_o,r_o}$. More research is needed to identify these latter idempotents. Although $\Dcc_g^{l_o,r_o}$ sometimes has idempotents, it is very common for $\Dcc_g^{l_o,r_o}$ not to have idempotents. For example, only two of the twenty-four $\Dc$-classes of $\ftt$ corresponding to elements $g$ of height $4$ have idempotents in $\Dcc_g^{l_o,r_o}$.

\section{Some considerations for future research}

The structure of the semigroups $\ftx$ is not yet fully understood and more research is needed. For example, as it becomes clear from the previous section, the identification of idempotents is not immediate, which makes it difficult to understand the structure of the biordered set of idempotents of $\ftx$. Also, working with triples nested inside triples is not an easy task. So, the pursue of better, alternative ways to represent and work with the elements of $\ftx$ is advisable. There are also decidability and other questions, some of them stated at the end of Section 5, related to the fact that all regular semigroups weakly generated by $|X|$ idempotents are homomorphic images of $\ftx$. Such questions can now be addressed. 

Dolinka and Ru\v{s}kuc \cite{doru13} proved that every finitely presented group $G$ (which includes all finite groups) is a maximal subgroup of a free regular idempotent generated semigroup $RIG(B_G)$ on some finite band $B_G$. Note that $RIG(B_G)$ is a homomorphic image of $\ft_k$ for $k=|B_G|$ since $RIG(B_G)$ is generated by the idempotents of the biordered set $B_G$. Thus $RIG(B_G)$ also strongly divides $\ftt$. One can ask now which finitely presented groups and which finite groups appear as maximal subgroups of free regular idempotent generated semigroups also weakly generated by $k$ idempotents. In particular, one can investigate (i) what kind of finite groups can we get as maximal subgroups of free regular idempotent generated semigroups weakly generated by two idempotents? (ii) Can we get all of them? If not, we may consider a hierarchy of classes of finite groups, each class defined by the finite groups which are maximal subgroups of free regular idempotent generated semigroups weakly generated by at most $k$ idempotents.

To address the previous questions it will be important to study the homomorphic images of $\ftt$. In particular, we should investigate which semigroups are regular subsemigroups of homomorphic images of $\ftt$. Note that we have already proved that each weakly finitely idempotent generated semigroups strongly divides $\ftt$, that is, it is a homomorphic image of a regular subsemigroup of $\ftt$. Thus, it is natural to conjecture that all weakly finitely idempotent generated semigroups are regular subsemigroups of homomorphic images of $\ftt$.  A positive answer to this conjecture is related to a positive answer to the following question: can any surjective homomorphism $\varphi:T\to S$, where $T$ is a regular subsemigroup of $\ftt$, be extended to a homomorphism $\phi:\ftt\to S_1$? Can we choose $S_1$ to be weakly generated by two idempotents? Note that not all homomorphic images of $\ftt$ are weakly generated by two idempotents (see Example 1).

A different line of research is the attempt to generalize the results presented here for the non-idempotent case. In other words, is there a regular semigroup $F(X)$ weakly generated by a set $X$ such that all other regular semigroups weakly generated by $X$ are homomorphic images of $F(X)$? Of course, we should start with the case $X=\{x\}$, but we should not expect a simple structure for $F(\{x\})$, if $F(\{x\})$ exists. Most certainly, $F(\{x\})$ will contain a copy of $\ftt$ weakly generated by the set of idempotents $\{xx',x'x\}$.

The generalization of the results obtained here for the non-idempotent case may be of interest for the theory of e-varieties of regular semigroups. This theory began in the 1990s \cite{ha1,ks1} and a great effort was made on the development of a Birkhoff-type theorem for e-varieties of regular semigroups. Unfortunately, only partial results were found, namely for the e-varieties of locally inverse semigroups \cite{au94,au95} and for the e-varieties of regular $E$-solid semigroups \cite{ks2}, and the interest on general e-varieties of regular semigroups diminished considerably. These partial results were based on the concepts of `bifree objects' and `biequational classes'.  

Now, if one can show that $F(X)$ exists, then one should explore if its universal properties make it a good candidate for the development of a Birkhoff-type theorem that works for all e-varieties of regular semigroups. In this regard, we should point out that, for e-varieties $\bf V$ with bifree objects on a set $X$, the concept of bifree object could be alternatively defined as follows: it is a semigroup $F{\bf V}(X)\in{\bf V}$ weakly generated by $X$ such that all mappings $\theta:X\to S\in{\bf V}$ can be extended to a homomorphism $\varphi:F{\bf V}(X)\to S$.
Concerning the usual concept of bifree object, we add the condition of $F{\bf V}(X)$ being weakly generated by $X$, but drop the necessity of the mappings $\theta$ being matched and the uniqueness in the homomorphism extensions $\varphi$.

\vspace*{.5cm}

\noindent{\bf Acknowledgments}: This work was partially supported by CMUP (UID/ MAT/00144/2019), which is funded by FCT with national (MCTES) and European structural funds through the programs FEDER, under the partnership agreement PT2020. The author also would like to acknowledge the importance of the GAP software \cite{gap}, and its Semigroup package \cite{mitchell}, in the research presented in this paper: the several simulations computed in GAP allowed the emergence of the pattern used to construct the set $\Gx$.

\end{document}